\documentclass[reqno,10pt]{amsart}

\usepackage[left=3cm,right=3cm,top=3cm,bottom=3cm]{geometry}
\usepackage{amssymb,bm}
\usepackage{mathrsfs}
\usepackage{xcolor}
\usepackage{enumitem}
\usepackage{appendix}
\usepackage{diffcoeff}
\usepackage{amsmath,cases}
\usepackage{empheq}

\newtheorem{theorem}{Theorem}[section]
\newtheorem{lemma}[theorem]{Lemma}
\newtheorem{proposition}[theorem]{Proposition}
\newtheorem{corollary}[theorem]{Corollary}

\theoremstyle{definition}

\theoremstyle{remark}
\newtheorem{remark}[theorem]{Remark}

\numberwithin{equation}{section}

\DeclareMathOperator{\dive}{div}
\DeclareMathOperator{\curl}{curl}

\begin{document}

\title[Two-parameter relaxation of the incompressible Navier-Stokes equations]{Convergence of a two-parameter hyperbolic relaxation system toward the incompressible Navier-Stokes equations}



\author[Qian Huang]{Qian Huang}
\address{Institute of Applied Analysis and Numerical Simulation, University of Stuttgart, 70569 Stuttgart, Germany}
\email{qian.huang@mathematik.uni-stuttgart.de; hqqh91@qq.com}

\author[Christian Rohde]{Christian Rohde}
\address{Institute of Applied Analysis and Numerical Simulation, University of Stuttgart, 70569 Stuttgart, Germany}
\email{christian.rohde@mathematik.uni-stuttgart.de}

\author[Ruixi Zhang]{Ruixi Zhang}
\address{Department of Mathematical Sciences, Tsinghua University, Beijing 100084, China}
\email{zhangrx24@mails.tsinghua.edu.cn; 1553548358@qq.com}



\date{}

\dedicatory{}

\begin{abstract}
We investigate a two-parameter hyperbolic relaxation approximation to the incompressible Navier-Stokes equations, incorporating a first-order relaxation and the artificial compressibility method. With vanishingly small perturbations of initial velocity, we rigorously prove the simultaneous convergence of fluid velocity and pressure toward the Navier-Stokes limit in the three-dimensional case by constructing an intermediate affine system to obtain the necessary error estimates for the pressure. Furthermore, we extend the velocity convergence analysis to the case of $\mathcal O(1)$ initial velocity perturbations, and establish the global-in-time recovery of the velocity field using a modulated energy structure and delicate bootstrap arguments in both two- and three-dimensional settings.
\end{abstract}

\maketitle

\section{Introduction}

Consider the incompressible Navier-Stokes equations
\begin{subequations} \label{eq:ns}
\begin{empheq}[left=\empheqlbrace]{align}
  &\nabla\cdot u = 0, \label{eq:nsa} \\
  &\partial_t u + \nabla \cdot (u \otimes u) = -\nabla p + \Delta u, \label{eq:nsb} \\
  &u(0,\cdot) =u_0 \label{eq:nsc}
\end{empheq}
\end{subequations}
for $(t,x)\in [0,T_0]\times \mathbb T^d$, where $T_0\in (0,\infty]$ denotes the lifespan of the solution and $\mathbb T^d=\mathbb R^d / \mathbb Z^d$ ($d=2,3$) is the $d$-dimensional unit periodic square. The unknowns $u(t,x) \in \mathbb{R}^d$ and $p(t,x) \in \mathbb{R}$ represent the velocity field and the scalar pressure of the fluid, respectively. We impose the condition $\int_{\mathbb{T}^d} p(t,x) \, dx = 0$ for all $t$ to ensure the uniqueness of the pressure. The viscosity coefficient is set as unity. By introducing the stress-like tensor $U = u\otimes u -\nabla u \in \mathbb R^{d\times d}$, the momentum equation can be rewritten as
\[
  \partial_t u + \nabla \cdot U + \nabla p = 0.
\]
For convenience, let $p_0=p(0,\cdot)$ and $U_0=U(0,\cdot)$. In the two-dimensional case $\mathbb T^2$, it is a classical result that a unique smooth solution exists globally for any smooth initial data \cite{kato1972,ladyzhenskaya1969}. In contrast, for {$\mathbb T^3$}, while the global existence of Leray-Hopf weak solutions is well established \cite{leray1934}, only the local existence of smooth solutions on a finite time interval $[0, T_0]$ is guaranteed for arbitrary smooth $u_0$ in $\mathbb{T}^3$. Currently, the global existence of smooth solutions in 3D ($T_0=\infty$) can only be ensured with sufficiently small initial data, either in smooth settings or in critical spaces such as $H^{1/2}(\mathbb{T}^3)$ or $L^3(\mathbb{T}^3)$ \cite{Fujita1964,kato1984}. 

Driven by the foundational role of (\ref{eq:ns}) in fluid mechanics, various approximation strategies have been developed to address its inherent mathematical challenges. A significant paradigm involves the hydrodynamic limit of kinetic equations, where (\ref{eq:ns}) is recovered from the Boltzmann equation under diffusive scaling \cite{de1989CPAM, golse2004}. This theoretical bridge motivates the use of relaxation models, such as velocity-discrete BGK (Bhatnagar-Gross-Krook) systems \cite{bianchini2019,JY2003}, which effectively replace the global constraint of the pressure with localized evolution laws \cite{aregba2000,jinxin1995}.

Within this framework, we focus on a new approximation to (\ref{eq:ns}) on $(t,x)\in [0,T_{\epsilon,\delta})\times \mathbb T^d$, proposed in \cite{HRYZ25} as a first-order hyperbolic relaxation system:
\begin{subequations} \label{eq:relax}
\begin{empheq}[left=\empheqlbrace]{align}
  &\epsilon \partial_t p^{\epsilon,\delta} + \nabla \cdot u^{\epsilon,\delta} =0, \label{eq:relax_a} \\
  &\partial_t u^{\epsilon,\delta} + \nabla \cdot U^{\epsilon,\delta} + \nabla p^{\epsilon,\delta} = 0, \label{eq:relax_b} \\
  &\delta \partial_t U^{\epsilon,\delta} + \nabla u^{\epsilon,\delta} = u^{\epsilon,\delta}\otimes u^{\epsilon,\delta} - U^{\epsilon,\delta}, \label{eq:relax_c} \\
  &u^{\epsilon,\delta}(0,\cdot)=u^{\epsilon,\delta}_0, \ p^{\epsilon,\delta}(0,\cdot)=p^{\epsilon,\delta}_0, \ U^{\epsilon,\delta}(0,\cdot)=U^{\epsilon,\delta}_0,
\end{empheq}
\end{subequations}
where $\epsilon,\delta>0$ are two small parameters. Here $T_{\epsilon,\delta}$ denotes the lifespan of the solution $(p^{\epsilon,\delta},u^{\epsilon,\delta},U^{\epsilon,\delta})\in \mathbb R\times \mathbb R^d \times \mathbb R^{d\times d}$. Clearly, (\ref{eq:relax}) formally recovers (\ref{eq:ns}) as both parameters $\epsilon$ and $\delta$ vanish. 
This formulation can be viewed as a synergistic extension of two classical approximation paradigms. On one hand, it incorporates the artificial compressibility (AC) method \cite{chorin1967jcp,chorin1967ams,chorin1968,temam1969I,temam1969II}, which relaxes the rigid divergence-free constraint through the evolution equation (\ref{eq:relax_a}), thereby circumventing the computational burden associated with the global Poisson equation for pressure. On the other hand, the stress-like tensor $U^{\epsilon,\delta}$ follows the relaxation dynamics (\ref{eq:relax_c}), a structure inspired by the first-order model proposed in \cite{Brenier2003} via the diffusive scaling:
\begin{equation} \label{eq:relax_para}
\left\{ \
\begin{aligned}
  &\nabla\cdot u^\delta = 0, \\
  &\partial_t u^\delta + \nabla \cdot U^\delta + \nabla p^\delta = 0, \\
  &\delta \partial_t U^\delta + \nabla u^\delta = u^\delta\otimes u^\delta - U^\delta, \\
  &u^\delta(0,\cdot)=u_0^\delta, \ U^\delta(0,\cdot) = U_0^\delta,
\end{aligned}
\right.
\end{equation}
with $\delta>0$ and $U^\delta: (t,x) \mapsto \mathbb R^{2\times 2}$. 
It is worth mentioning that the convergence of $u^\delta$ in (\ref{eq:relax_para}) toward a solution $u$ of (\ref{eq:ns}) has been established in various functional settings \cite{Brenier2003,hachicha2014, ilyin2018,paicu2007,racke2012}. 
By integrating these two approaches, the fully hyperbolic system (\ref{eq:relax}) ensures finite propagation speeds for all physical disturbances, aligning more closely with the principles of causality. It also becomes compatible with the rich library of numerical schemes for hyperbolic conservation systems, thus benefiting multiple downstream tasks like statistics of turbulent flows \cite{friedrich2012,Lundgren}.

The purpose of this work is to provide rigorous convergence analyses of the two-parameter approximation system (\ref{eq:relax}) toward the Navier-Stokes system (\ref{eq:ns}) on $\mathbb T^d$ ($d=2,3$). Such a theoretical foundation is indispensable for validating the application of the model in numerical simulations. Note that the existence of the solutions to (\ref{eq:relax}) is guaranteed by the well-established theory \cite{kato1975}. While one-parameter limits for relaxation or artificial compressibility models have been extensively documented \cite{klainerman1981, temam2001ns, Yong1999}, the simultaneous two-parameter limit $(\epsilon, \delta) \to 0$ for \eqref{eq:relax} remains far less explored. We refer to \cite{brenier2005,xu2009,Yong2001} for some related works on two-parameter limit problems.

For the analysis of (\ref{eq:relax}), the first technical challenge is to derive a convergence result for the pressure, which requires a simultaneous control of errors for both velocity and pressure. While it seems natural to pass to the limit in a way as $(u, p)^{\epsilon, \delta} \to (u, p)^{\epsilon \text{ or } \delta} \to (u, p)$, this sequential approach turns out to be unfeasible: Existing studies on (\ref{eq:relax_para}) never show the convergence of $p^\delta$, hence providing no useful error rates for ``dropping $\delta$” \cite{Brenier2003,hachicha2014,paicu2007}; For the analysis of AC in the limit $\epsilon\to 0$, the convergence of pressure can be ensured, but it generally lacks the quantitative error bounds essential for our convergence analysis \cite{donatelli2010,kagei2021,temam2001ns}. We recently overcome this barrier, for the two-dimensional case, by introducing an intermediate affine system \cite{HRYZ25}. In this novel approach, the discrepancies of the auxiliary system with (\ref{eq:ns}) and (\ref{eq:relax}) can both be bounded, thereby leading to convergence results of the pressure. Our first observation of the present work is that this methodology can be effectively extended to the three-dimensional setting $\mathbb{T}^3$. Although the transition to 3D introduces additional complexity, most notably that the vorticity becomes a vector field, the basic energy-based techniques remain applicable.

However, all the existing results in \cite{HRYZ25} require vanishingly small amplitude of initial velocity perturbations $|u^{\epsilon,\delta}_0-u_0|$, typically of order $\mathcal O((\epsilon+\delta)^\alpha)$ with $\alpha>0$. A more fundamental challenge arises when the initial velocity perturbation is of order $\mathcal O(1)$, a regime that is crucial for ensuring the robustness of model (\ref{eq:relax}) against significant initial inconsistencies often encountered in numerical start-ups. The transition from $o(1)$ to $\mathcal O(1)$ perturbations renders the conventional energy-based arguments insufficient. The energy estimates in $H^s$ ($s\ge 2$), leveraging the symmetric hyperbolic structure of (\ref{eq:relax}), typically yield a guaranteed existence time that scales inversely with the $L^\infty$-norm of the initial perturbation. Consequently, unlike the small-perturbation case where the lifespan can be extended to match $T_0$ of (\ref{eq:ns}) (as revealed for $\mathbb{T}^3$ in the results of Section~\ref{subsec:main_small}), the $\mathcal O(1)$ perturbation limits the convergence to a significantly shorter time interval. This limitation stems from the fact that the dissipation in (\ref{eq:relax}) is solely provided by the damping of $U^{\epsilon,\delta}$. Since the relation between $U^{\epsilon,\delta}$ and $u^{\epsilon,\delta}$ is only recovered in the limit, this damping mechanism is not powerful enough to suppress nonlinear effects.

To fill the gap in the $\mathcal O(1)$ regime, we develop a refined analysis by invoking the modulated energy functional introduced in \cite{Brenier2003}. This approach provides a robust dissipation mechanism involving $\nabla \omega^{\epsilon,\delta}$ (with $\omega^{\epsilon,\delta} = \curl u^{\epsilon,\delta}$), which remains effective in both two and three dimensions. While the $L^2$-convergence of the velocity can be managed through the more standard energy estimates, the control of the curl part becomes the decisive factor for establishing higher-order convergence. However, the presence of the small compressibility in our system (\ref{eq:relax_a}) introduces an additional layer of complexity that is absent in the original divergence-free model (\ref{eq:relax_para}) in \cite{Brenier2003}. To resolve this issue, we apply the natural Hodge-type decomposition and implement a carefully-designed bootstrap argument that works effectively for the regime where $\epsilon = o(\delta)$ and the initial divergence of velocity is sufficiently small relative to the initial curl. This allows us to rigorously establish the global-in-time velocity convergence in the $\mathcal O(1)$ regime over the entire lifespan $T_0$ of the limit solution. Specifically, for any $T < T_0$ (including $T_0 = \infty$), $u^{\epsilon,\delta}$ exists and approaches $u$ on $[0, T]$ as $\epsilon$ and $\delta$ vanish under the appropriate scaling. Notably, we reveal a structural richness on $\mathbb T^3$ by identifying two distinct ``extreme cases" for the convergence rates, while for $\mathbb T^2$, only one such construction exists.

The remainder of the paper is organized as follows. The main results and preparations are collected in Section~\ref{sec:main}. Section~\ref{sec:small} presents the proofs of Theorem~\ref{thm:uconverg_H1}, Corollary~\ref{cor:convergence} and Theorem~\ref{thm:pressure_3d} in the vanishing-initial-perturbation regime on $\mathbb T^3$. Section~\ref{sec:large} is devoted to the proofs of Theorems~\ref{thm:critical_3d}, \ref{thm:critical_3d_inhomo} and \ref{thm:critical_2d} in the more challenging, $\mathcal O(1)$-initial-perturbation regime. Conclusions are drawn in Section \ref{sec:concl}.

\textbf{Notations.} All positive constants, like $C$, $C'$, $C_T$, etc, will be clarified in each statement and proof on their dependence and whether they are generic or fixed. In any case, these constants do not depend on $(\epsilon,\delta)$. The standard $L^p$- ($p=2,3,4,6,\infty$) and $H^s$- ($s=1,2,3$) norms are both for $\mathbb T^d$ (the value of $d$ is specified in each result and proof). We adopt the following shorthand and norm definitions for the continuous-time Sobolev spaces $C^k H^s = C^k([0,T]; H^s(\mathbb{T}^d))$:
\[
    \|u\|_{C^0 H^s} := \sup_{t \in [0, T]} \|u(t, \cdot)\|_{H^s}, \quad
    \|u\|_{C^0 H^s \cap C^1 H^{s-1}}^2 := \sup_{t \in [0, T]} \left( \|u(t, \cdot)\|_{H^s}^2 + \|\partial_t u(t, \cdot)\|_{H^{s-1}}^2 \right).
\]
For the differential operators, we denote the gradient of a vector field $u = (u_1, \dots, u_d)^\top$ by the matrix $(\nabla u)_{ij} = \partial_{x_j} u_i$, and $\nabla^2 u$ represents the tensor of all second-order partial derivatives. The divergence of a vector field is defined as $\text{div} \, u = \nabla \cdot u = \sum_{i=1}^d \partial_{x_i} u_i$, and $\Delta = \sum_{i=1}^d \partial_{x_i}^2$ denotes the Laplacian. For $d=3$, $\text{curl} \, u = \nabla \times u$, while for $d=2$, it is understood as the scalar $\partial_{x_1} u_2 - \partial_{x_2} u_1$. For any two matrices $A, B \in \mathbb{R}^{d \times d}$, $A \otimes B$ denotes the tensor product, and $A : B = \sum_{i,j} A_{ij} B_{ij}$ represents the Frobenius inner product. Finally, we use the standard Landau notation $f = \mathcal{O}(g)$ to denote that $|f| \leq C|g|$ as the parameters $\epsilon, \delta$ vanish, where $C>0$ is a generic constant independent of $(\epsilon,\delta)$.

\section{Main results} \label{sec:main}
In this section, we present our main results. It is convenient to consider a smooth solution $(p,u,U)$ of (\ref{eq:ns}) that exists on $[0,T_0]$ (with $T_0\le \infty$). Let 
$(u^{\epsilon,\delta},p^{\epsilon,\delta},U^{\epsilon,\delta})$ be a series of solutions of the initial-value problem (\ref{eq:relax}) on $[0,T_{\epsilon,\delta})\times \mathbb T^d$ with smooth initial data $(u_0^{\epsilon,\delta},p_0^{\epsilon,\delta},U_0^{\epsilon,\delta})$, indexed with sufficiently small $\epsilon$ and $\delta$. Our main goal is to prove convergence results $u^{\epsilon,\delta}\to u$ as $\epsilon,\delta\to 0$, in some proper sense, under perturbations of initial data. Most of the results are derived for three-dimensional flows on $\mathbb T^3$, and only Theorem~\ref{thm:critical_2d} is stated for two-dimensional flows on $\mathbb T^2$.

\subsection{Vanishing amplitude of initial velocity perturbations on $\mathbb T^3$} \label{subsec:main_small}

Consider the regime where
\begin{equation} \label{eq:small_Linf} 
    \|u_0^{\epsilon,\delta}-u_0 \|_{L^\infty(\mathbb T^3)}=\mathcal O((\epsilon+\delta)^\alpha) \quad \text{with} \quad \alpha>0.
\end{equation}
In this case, we can verify the (uniformly) local existence of $u^{\epsilon,\delta}$ on $[0,T]$, with $T>0$ independent of $\epsilon,\delta$, and the convergence $u^{\epsilon,\delta}\to u$ in $L^\infty({[0,T]};H^1(\mathbb T^3))$. Precisely, we have the following.

\begin{theorem} \label{thm:uconverg_H1}
    Assume that there exist constants $a \in [0,2)$ and $C_0>0$, both independent of $\epsilon$ and $\delta$, such that
    \begin{subequations} \label{eq:uH1_assump}
    \begin{align}
        &{\| u_0^{\epsilon,\delta}-u_0\|}_{H^1(\mathbb T^3)}^2 + (\epsilon+\delta)^a \left\|\nabla^2 \left(u_0^{\epsilon,\delta}-u_0\right)\right\|_{L^2(\mathbb T^3)}^2 \le C_0(\epsilon+\delta), \label{eq:H1_assump_u} \\
        &\left\| \left(p_0^{\epsilon,\delta},U_0^{\epsilon,\delta}\right) \right\|_{H^1(\mathbb T^3)}^2 + (\epsilon+\delta)^a \left\| \nabla^2\left(p_0^{\epsilon,\delta},U_0^{\epsilon,\delta}\right) \right\|_{L^2(\mathbb T^3)}^2 \le C_0.
    \end{align}
    \end{subequations}
    Then, if $T_0<\infty$, the solution $(u^{\epsilon,\delta},p^{\epsilon,\delta},U^{\epsilon,\delta})$ of the initial value problem for \eqref{eq:relax} exists on $[0,T_0]$ for sufficiently small $\epsilon$ and $\delta$.
    Otherwise, if $T_0=\infty$, the existing time $T_{\epsilon,\delta}$ is uniformly bounded from below and satisfies
    \begin{equation} \label{eq:Texist}
        T_{\epsilon,\delta} \ge C'(2-a)|\log(\epsilon+\delta)| \to \infty \quad \text{as} 
        \quad \epsilon+\delta\to 0,
    \end{equation}
    where $C'$ depends only on $C_0$ and the smooth solution $(p,u,U)$. 
    Moreover, for any $T>0$ such that both $u^{\epsilon,\delta}$ and $u$ exist on $[0,T]$, we have
    \begin{equation} \label{eq:u_converg}
        \sup_{t\in[0,T]} {\| u^{\epsilon,\delta}(t,\cdot)-u(t,\cdot)\|}_{H^1(\mathbb T^3)}^2 \le C_T(\epsilon+\delta)
    \end{equation}
    and
    \begin{equation} \label{eq:u_Linf_bound}
       \sup_{t\in[0,T]} \| u^{\epsilon,\delta}(t,\cdot)\|_{L^\infty(\mathbb T^3)} \le C_T    
    \end{equation}
    for sufficiently small $\epsilon+\delta \ll 1$, where $C_T$ depends only on $C_0$, $T$ and $(p,u,U)$.
\end{theorem}

The proof of Theorem~\ref{thm:uconverg_H1} is presented in Section~\ref{subsec:thm:uconverg_H1}, which will be heavily relied on the following interpolation inequalities on $\mathbb T^3$. These results are standard \cite{Adams2003,temam2001ns} and their proofs are omitted.

\begin{lemma} \label{lem:interp_ineq}
    Let  $w\in H^2(\mathbb T^3)$. Then there is a constant $C\ge 0$ that depends on the domain $\mathbb T^3$ only such that the inequalities 
    \begin{subequations}
    \begin{align}
        {\|w\|}_{L^3} &\le C \|w\|_{L^2}^{\frac{1}{2}} \|w\|_{H^1}^{\frac{1}{2}}, \label{eq:interp_L3} \\
        {\|w\|}_{L^4} &\le C \|w\|_{L^2}^{\frac{1}{4}} \| w\|_{H^1}^{\frac{3}{4}}, \\
        {\|w\|}_{L^\infty} &\le C \|w\|_{H^1}^{\frac{1}{2}} \|w\|_{H^2}^{\frac{1}{2}}, \label{eq:interp_inf}\\
        {\|\nabla w\|}_{L^2} &\le \|w\|_{L^2}^{\frac{1}{2}} \|\nabla^2 w\|_{L^2}^{\frac{1}{2}}, \label{eq:interp_nabla_L2} \\
        {\|\nabla w\|}_{L^4}^2 &\le C\|w\|_{L^\infty} \|\nabla^2 w\|_{L^2}, \label{eq:interp_d}
    \end{align}
    \end{subequations}
   hold.
\end{lemma}

It is thus seen from (\ref{eq:interp_inf}) that the assumption (\ref{eq:H1_assump_u}) in Theorem~\ref{thm:uconverg_H1} truly falls into the regime (\ref{eq:small_Linf}):  $\|u^{\epsilon,\delta}_0-u_0\|_{L^\infty}^2\le C\|u^{\epsilon,\delta}_0-u_0\|_{H^1} \|u^{\epsilon,\delta}_0-u_0\|_{H^2}\le C(\epsilon+\delta)^{1-a/2}$ with $a<2$.
Note that the constraint $a<2$ is essential here for a strictly positive lower bound of the lifespan $T_{\epsilon,\delta}$ in (\ref{eq:Texist}).

A similar result for $u^{\epsilon,\delta}\to u$ in $L^\infty([0,T];L^2(\mathbb T^3))$ (given a proper $T>0$) can be stated as follows.

\begin{corollary} \label{cor:convergence}    
Assume that there exist constants $a \in [0,\frac{4}{3})$ and $C_0>0$, both independent of $\epsilon$ and $\delta$, such that
    \begin{subequations}
    \begin{align}
        &{\| u_0^{\epsilon,\delta}-u_0\|}_{L^2(\mathbb T^3)}^2 + (\epsilon+\delta)^a \left\|\nabla^2 \left(u_0^{\epsilon,\delta}-u_0\right)\right\|_{L^2(\mathbb T^3)}^2 \le C_0(\epsilon+\delta), \\
        &\left\| \left(p_0^{\epsilon,\delta},U_0^{\epsilon,\delta}\right) \right\|_{L^2(\mathbb T^3)}^2 + (\epsilon+\delta)^a \left\| \nabla^2\left(p_0^{\epsilon,\delta},U_0^{\epsilon,\delta}\right) \right\|_{L^2(\mathbb T^3)}^2 \le C_0.
    \end{align}
    \end{subequations}
    Then, if $T_0<\infty$, the solution $(u^{\epsilon,\delta},p^{\epsilon,\delta},U^{\epsilon,\delta})$ exists on $[0,T_0]$ for sufficiently small $\epsilon$ and $\delta$. Otherwise, if $T_0=\infty$, the existing time $T_{\epsilon,\delta}$ is uniformly bounded from below and satisfies
    \[
        T_{\epsilon,\delta} \ge C' (2-a) |\log(\epsilon+\delta)| \to \infty \quad \text{as} 
        \quad \epsilon+\delta\to 0,
    \]
    where $C'$ depends only on $C_0$ and the smooth solution $(p,u,U)$.
    Moreover, for any $T>0$ such that both $u^{\epsilon,\delta}$ and $u$ exist on $[0,T]$, we have
    \[
        \sup_{t\in[0,T]} {\| u^{\epsilon,\delta}(t,\cdot)-u(t,\cdot)\|}_{L^2(\mathbb T^3)}^2 \le C_T(\epsilon+\delta)
    \]
    and
    \[
       \sup_{t\in[0,T]} \| u^{\epsilon,\delta}(t,\cdot)\|_{L^\infty(\mathbb T^3)} \le C_T  
    \]
    for sufficiently small $\epsilon+\delta \ll 1$, where $C_T$ depends only on $C_0$, $T$ and $(p,u,U)$.
\end{corollary}

The proof of Corollary~\ref{cor:convergence} is given in Section~\ref{subsec:cor:convergence}. Note that in this case, it holds that $\|u^{\epsilon,\delta}_0-u_0\|_{H^1}^2 \le C(\epsilon+\delta)^{1-a/2}$ due to (\ref{eq:interp_nabla_L2}), and hence $\|u^{\epsilon,\delta}_0-u_0\|_{L^\infty}^2 \le C(\epsilon+\delta)^{1-3a/4}$ with $3a/4<1$. Therefore, it still belongs to the regime (\ref{eq:small_Linf}), but has different orders of $(\epsilon+\delta)$ as compared with (\ref{eq:H1_assump_u}).

\begin{remark}
The results of Theorem~\ref{thm:uconverg_H1} and Corollary~\ref{cor:convergence} show that the approximate system (\ref{eq:relax}) is well chosen in the following sense. For sufficiently small $\epsilon,\delta$, the maximal life span $T_{\epsilon,\delta}$ of solutions to (\ref{eq:relax}) is bounded from below by the life span $T_0<\infty$ of the solutions to the original initial value problem (\ref{eq:ns}) for the Navier-Stokes system. This is relevant for practical computations that employ a numerical method  for (\ref{eq:relax}) to approximate the Navier-Stokes system and use small but finite values for $\epsilon$ and $\delta$.
\end{remark}

Under the small $L^\infty(\mathbb T^d)$-perturbation (of the initial velocity) regime, it is even possible to obtain a convergence result on the pressure sequence $p^{\epsilon,\delta}$ in $L^\infty([0,T];H^1(\mathbb T^d))$. This is more challenge since a direct energy method in the proof of Theorem~\ref{thm:uconverg_H1} can at best yield $\sup_{t\in[0,T]}\| (p^{\epsilon,\delta}-p)(t,\cdot) \|_{L^2}=\mathcal O(1)$ (see Section~\ref{subsec:thm:uconverg_H1}). Therefore, we need to work with stricter requirements on the $H^1(\mathbb T^d)$-norm of the initial-velocity perturbation and an additional constraint 
$\delta=o(\sqrt\epsilon)$. A result for $\mathbb T^2$ has been stated in Theorem 2.4 of \cite{HRYZ25}, whereas for $\mathbb T^3$, we now have the following.

\begin{theorem} \label{thm:pressure_3d} 
Assume that there exist constants $a\in [0,1)$ and $C_0>0$, both independent of $\epsilon$ and $\delta$, such that
    \begin{align}
        & \delta \le C_0\sqrt\epsilon, \label{eq:p_scaling} \\
        &\| u^{\epsilon,\delta}_0 - u_0\|_{H^1(\mathbb T^3)} + \sqrt\epsilon \|p^{\epsilon,\delta}_0 - p_0\|_{H^1(\mathbb T^3)}+ \sqrt\delta \|U^{\epsilon,\delta}_0 - U_0\|_{H^1(\mathbb T^3)} \le C_0(\epsilon+\delta), \label{eq:3dp_assump_1} \\
        &\|\nabla \dive u^{\epsilon,\delta}_0\|_{L^2(\mathbb T^3)} + \delta \left(\| \nabla^2 p^{\epsilon,\delta}_0\|_{L^2(\mathbb T^3)} + \|\nabla \dive U^{\epsilon,\delta}_0\|_{L^2(\mathbb T^3)} \right) \le C_0(\epsilon+\delta), \label{eq:3dp_assump_2} 
    \end{align}
    and
    \begin{equation} \label{eq:3dp_assump_3}
        {\|\nabla^2(u^{\epsilon,\delta}_0, \sqrt\epsilon p^{\epsilon,\delta}_0, \sqrt\delta U^{\epsilon,\delta}_0)\|}_{L^2(\mathbb T^3)} \le C_0(\epsilon+\delta)^{-a}.
    \end{equation}
    Then, for any $T>0$ such that both $u^{\epsilon,\delta}$ and $u$ exist on $[0,T]$, we have
    \[
        \sup_{t\in[0,T]} \left( {\| u^{\epsilon,\delta}(t,\cdot)-u(t,\cdot)\|}_{H^1(\mathbb T^3)} + \sqrt\epsilon \|p^{\epsilon,\delta}(t,\cdot)-p(t,\cdot)\|_{H^1(\mathbb T^3)} \right) \le C_T(\epsilon+\delta)
    \]
    for $\epsilon+\delta\ll 1$, where $C_T$ depends only on $C_0$, $a$, $T$, and the smooth solution $(p,u,U)$. In particular, if $\delta=o(\sqrt\epsilon)$ (that is, $\delta/\sqrt\epsilon \to 0$ as $\epsilon\to 0$), 
    this implies
    \[ 
        \sup_{t\in[0,T]} \|p^{\epsilon,\delta}(t,\cdot)-p(t,\cdot)\|_{H^1(\mathbb T^3)} \to 0 \quad \text{as} \quad \epsilon\to 0.
    \]
\end{theorem}

The proof of (\ref{thm:pressure_3d}) is given in Section~\ref{subsec:thm:pressure_3d}, where the technique of an intermediate affine system, firstly developed in our previous work \cite{HRYZ25}, is successfully extended to the 3-D condition. Here we point out that $\|u^{\epsilon,\delta}_0-u_0\|_{L^\infty}^2\le C(\epsilon+\delta)^{1-a}$ (with $a<1$) holds by (\ref{eq:3dp_assump_1}), (\ref{eq:3dp_assump_3}) and (\ref{eq:interp_inf}), thus again verifying the regime (\ref{eq:small_Linf}).

\subsection{Order-one initial velocity perturbations}

Now we move on to the much more difficult regime with larger perturbation
\begin{equation} \label{eq:O1_Linf} 
    \|u_0^{\epsilon,\delta}-u_0 \|_{L^\infty(\mathbb T^d)}=\mathcal O(1).
\end{equation}
The goal is still to show that, as $\epsilon,\delta$ both tend to 0, the velocity sequence $u^{\epsilon,\delta}$ of (\ref{eq:relax}) converges to the smooth solution $u$ of (\ref{eq:ns}) in $L^\infty([0,T_0];H^1(\mathbb T^d))$. For the sake of simplicity, let us only consider the case $T_0<\infty$; nevertheless, if the limit solution $u$ of (\ref{eq:ns}) exists globally, $u^{\epsilon,\delta}$ exists and approaches $u$ on $[0, T]$ for any $T < \infty$ as $\epsilon$ and $\delta$ vanish under the appropriate scaling.

Under the 3-D condition (on $\mathbb T^3$), we indeed establish two types of convergence results, formulated below as Theorems~\ref{thm:critical_3d} and \ref{thm:critical_3d_inhomo}, respectively. Both only work under (slightly different) additional constraints that $\epsilon$ is smaller than $\delta$. The proofs will be presented in Sections~\ref{subsec:thm:critical_3d} and \ref{subsec:thm:critical_3d_inhomo}, respectively.

\begin{theorem} \label{thm:critical_3d}
 Assume that there exists a constant $C_0>0$ independent of $\epsilon$ and $\delta$ such that
\begin{subequations}\label{ass1}
\begin{align}
    &\epsilon \le C_0 \delta^2, \\
    &\|u^{\epsilon,\delta}_0 - u_0\|_{L^2(\mathbb T^3)}^2 \le C_0\delta^{\frac{3}{2}}, \\
    &\left\|\curl\left(u^{\epsilon,\delta}_0 - u_0\right)\right\|_{L^2(\mathbb T^3)}^2 + \delta\left\|\curl\left(u^{\epsilon,\delta}_0 - u_0\right)\right\|_{H^1(\mathbb T^3)}^2 \le C_0\delta^{\frac{1}{2}}, \label{eq:ass1_c}\\
    &\|\dive u^{\epsilon,\delta}_0\|_{L^2(\mathbb T^3)}^2 + \delta \|\dive u^{\epsilon,\delta}_0\|_{H^1(\mathbb T^3)}^2 \le C_0\delta^{-\frac{1}{2}}\epsilon,
\end{align}
\end{subequations}
and
\begin{subequations}\label{ass2}
\begin{align}
    &\|p^{\epsilon,\delta}_0\|_{H^1(\mathbb T^3)}^2 + \delta\|p^{\epsilon,\delta}_0\|_{H^2(\mathbb T^3)}^2 \le C_0\delta^{-\frac{1}{2}}, \\
    &{\|U^{\epsilon,\delta}_0\|}_{H^1(\mathbb T^3)}^2 + \delta \|\dive U^{\epsilon,\delta}_0\|_{H^1(\mathbb T^3)}^2 \le C_0\delta^{-\frac{1}{2}}. \label{eq:ass2_b} 
\end{align}
\end{subequations}   
Then, for sufficiently small $\delta\ll 1$, the solution $(u^{\epsilon,\delta},p^{\epsilon,\delta},U^{\epsilon,\delta})$ exists on $[0,T_0]$. Moreover, for $t\in[0,T_0]$, the following bounds hold true:
    \begin{subequations} \label{eq:bounds_critical}
    \begin{align}
        \|u^{\epsilon,\delta}(t,\cdot)-u(t,\cdot)\|_{L^2(\mathbb T^3)}^2&\le \mathcal C\delta^{\frac{3}{2}}, \label{eq:bounds_critical_u} \\
        {\|\curl (u^{\epsilon,\delta}-u)(t,\cdot)\|}_{L^2(\mathbb T^3)}^2 + \delta {\|\curl (u^{\epsilon,\delta}-u)(t,\cdot)\|}_{H^1(\mathbb T^3)}^2  &\le \mathcal C\delta^{\frac{1}{2}}, \label{eq:bounds_critical_curl} \\
        {\|\dive u^{\epsilon,\delta}(t,\cdot)\|}_{L^2(\mathbb T^3)}^2 + \delta {\|\dive u^{\epsilon,\delta}(t,\cdot)\|}_{H^1(\mathbb T^3)}^2 &\le \mathcal C\delta^{-\frac{1}{2}}\epsilon, \label{eq:bounds_critical_div}
    \end{align}
    \end{subequations}
where the constant $\mathcal C$ depends only on $C_0$, $T_0$ and the smooth solution $(p,u,U)$.
\end{theorem}

\begin{theorem} \label{thm:critical_3d_inhomo}
Assume that there exists a constant $C_0>0$ independent of $\epsilon$ and $\delta$ such that
\begin{subequations}
\begin{align}
    &\|u^{\epsilon,\delta}_0-u_0\|_{L^2(\mathbb T^3)}^2 \le C_0\delta^{\frac{1}{2}}, \\
    &\left\|\curl\left(u^{\epsilon,\delta}_0 - u_0\right)\right\|_{L^2(\mathbb T^3)}^2 + \delta\left\|\curl\left(u^{\epsilon,\delta}_0 - u_0\right)\right\|_{H^1(\mathbb T^3)}^2 \le C_0\delta^{\frac{1}{2}}, \\
    & \|\dive u^{\epsilon,\delta}_0\|_{L^2(\mathbb T^3)}^2 + \delta \|\dive u^{\epsilon,\delta}_0\|_{H^1(\mathbb T^3)}^2 \le C_0\delta^{-\frac{1}{2}}\epsilon,
\end{align}
\end{subequations}    
and
\begin{subequations}
\begin{align}
    &\|p^{\epsilon,\delta}_0\|_{H^1(\mathbb T^3)}^2 + \delta\|p^{\epsilon,\delta}_0\|_{H^2(\mathbb T^3)}^2 \le C_0\delta^{-\frac{1}{2}}, \\
    & \|U^{\epsilon,\delta}_0\|_{H^1(\mathbb T^3)}^2 + \delta \|\dive U^{\epsilon,\delta}_0\|_{H^1(\mathbb T^3)}^2 \le C_0\delta^{-\frac{1}{2}}.
\end{align}
\end{subequations}
Then, there exists a positive constant $\mu_*\ll 1$, which depends only on $C_0$, $T_0$ and the smooth solution $u$, such that for sufficiently small $\epsilon$ and $\delta$ satisfying
\begin{equation} \label{eq:relation_3d_inhomo}
    \epsilon \le \mu_* \delta,
\end{equation}
the solution $(u^{\epsilon,\delta},p^{\epsilon,\delta},U^{\epsilon,\delta})$ exists on $[0,T_0]$, and the following bounds hold true for $t\in[0,T_0]$:
\begin{subequations} \label{eq:bounds_critcoro}
\begin{align}
    &\|u^{\epsilon,\delta}(t,\cdot)-u(t,\cdot)\|_{L^2(\mathbb T^3)}^2\le \mathcal C\delta^{\frac{1}{2}}, \label{eq:bounds_critcoro_u} \\
    &\|\curl (u^{\epsilon,\delta}-u)(t,\cdot)\|_{L^2(\mathbb T^3)}^2 + \delta \|\curl (u^{\epsilon,\delta}-u)(t,\cdot)\|_{H^1(\mathbb T^3)}^2 \le \mathcal C\delta^{\frac{1}{2}}, \label{eq:bounds_critcoro_curl} \\
    &\|\dive u^{\epsilon,\delta}(t,\cdot)\|_{L^2(\mathbb T^3)}^2 + \delta \|\dive u^{\epsilon,\delta}(t,\cdot)\|_{H^1(\mathbb T^3)}^2 \le \mathcal C\delta^{-\frac{1}{2}}\epsilon, \label{eq:bounds_critcoro_div}
\end{align}
\end{subequations}
where the constant $\mathcal C$ depends only on $C_0$, $T_0$ and the smooth solution $u$.
\end{theorem}

The main difference between the two theorems lies in the order of the $L^2(\mathbb T^3)$-norm of the initial velocity perturbation $u^{\epsilon,\delta}_0-u_0$ with respect to small parameters  ($\delta^{3/4}$ vs. $\delta^{1/4}$), while the $H^1(\mathbb T^3)$- and $H^2(\mathbb T^3)$-norms of $u^{\epsilon,\delta}_0-u_0$ have the same order, being $\delta^{1/4}$ and $\delta^{-1/4}$, respectively (Note that the $H^1$- and $H^2$-norms must be recovered using (\ref{eq:hodge})). It is hence clear that (\ref{eq:O1_Linf}) holds in both cases. Nevertheless, the rationality for choosing these specific orders will be discussed in Remark~\ref{rem:orders}.

Our final result is an analogous convergence theorem for two-dimensional velocity $u^{\epsilon,\delta}$ of (\ref{eq:relax}) on $\mathbb T^2$. Let $(p,u,U)$ be a smooth solution of (\ref{eq:ns}) on $[0,T_0]\times \mathbb T^2$ with $T_0<\infty$. Unlike in the 3-D case, now we have only one choice of the order of initial perturbations; see a brief explanation in Remark~\ref{rem:orders_2d}. 

\begin{theorem} \label{thm:critical_2d}
Assume that there exists a constant $C_0\ge 1$ independent of $\epsilon$ and $\delta$ such that
\begin{subequations}
\begin{align}
    &\| u_0^{\epsilon,\delta}-u_0 \|_{L^2(\mathbb T^2)}^2 \le C_0\delta, \\
    &\|\curl u^{\epsilon,\delta}_0\|_{L^2(\mathbb T^2)}^2 + \delta \| \curl u^{\epsilon,\delta}_0 \|_{H^1(\mathbb T^2)}^2 \le C_0, \\
    &\|\dive u^{\epsilon,\delta}_0\|_{L^2(\mathbb T^2)}^2 + \delta \| \dive u^{\epsilon,\delta}_0 \|_{H^1(\mathbb T^2)}^2 \le C_0\delta^{-1}\epsilon,
\end{align}
\end{subequations}
and
\begin{subequations}
\begin{align}
    &\epsilon \|p^{\epsilon,\delta}_0\|_{L^2(\mathbb T^2)}^2 + \delta^2 \|p_0^{\epsilon,\delta}\|_{H^1(\mathbb T^2)}^2 + \delta^3 \|p_0^{\epsilon,\delta}\|_{H^2(\mathbb T^2)}^2 \le C_0\delta, \\ 
    &\|U^{\epsilon,\delta}_0\|_{L^2(\mathbb T^2)}^2 + \delta \| U^{\epsilon,\delta}_0 \|_{H^1(\mathbb T^2)}^2 + \delta^2 \|\dive U^{\epsilon,\delta}_0 \|_{H^1(\mathbb T^2)}^2 \le C_0.
\end{align}
\end{subequations}
Then, there exists a positive constant $\mu_*\ll 1$, which depends only on $C_0$, $T_0$ and the smooth solution $u$, such that for sufficiently small $\epsilon$ and $\delta$ satisfying
\begin{equation} \label{eq:relation_2d}
    \epsilon \le \mu_* \delta,
\end{equation}
the solution $(u^{\epsilon,\delta},p^{\epsilon,\delta},U^{\epsilon,\delta})$ exists on $[0,T_0]$, and the following bounds hold true for $t\in[0,T_0]$:
\begin{subequations} \label{eq:bounds_critcoro_2d}
\begin{align}
    &\|u^{\epsilon,\delta}(t,\cdot)-u(t,\cdot)\|_{L^2(\mathbb T^2)}^2\le \mathcal C\delta, \label{eq:bounds_crit2d_u} \\
    &\|\curl u^{\epsilon,\delta}(t,\cdot)\|_{L^2(\mathbb T^2)}^2 + \delta \|\curl u^{\epsilon,\delta}(t,\cdot)\|_{H^1(\mathbb T^2)}^2 \le \mathcal C, \label{eq:bounds_crit2d_curl} \\
    &\|\dive u^{\epsilon,\delta}(t,\cdot)\|_{L^2(\mathbb T^2)}^2 + \delta \|\dive u^{\epsilon,\delta}(t,\cdot)\|_{H^1(\mathbb T^2)}^2 \le \mathcal C\delta^{-1}\epsilon, \label{eq:bounds_crit2d_div}
\end{align}
\end{subequations}
where the constant $\mathcal C$ depends only on $C_0$, $T_0$ and the smooth solution $u$.
\end{theorem}

\subsection{Preliminaries}
In this subsection we collect the tools and lemmas that will be used in the proofs.
First, denote by
\begin{equation} \label{eq:res_def}
        v = u^{\epsilon,\delta}-u, \quad q=\sqrt\epsilon(p^{\epsilon,\delta}-p), \quad V=\sqrt\delta(U^{\epsilon,\delta}-U)
\end{equation}
rescaled error functions on $[0,T_{\epsilon,\delta})\times \mathbb T^d$. The system for $W=(q,v,V)$ reads as
\begin{equation} \label{eq:res}
    \left \{ \
    \begin{aligned}
    &\partial_t q + \frac{1}{\sqrt\epsilon} \nabla \cdot v = -\sqrt \epsilon \partial_t p, \\
    &\partial_t v + \frac{1}{\sqrt\delta} \nabla \cdot V + \frac{1}{\sqrt\epsilon} \nabla q = 0, \\
    &\partial_t V + \frac{1}{\sqrt\delta}\nabla v = \frac{2u v + v\otimes v}{\sqrt\delta} - \frac{V}{\delta} - \sqrt \delta \partial_t U,
    \end{aligned}
    \right.
\end{equation}
with $uv := (u\otimes v+v\otimes u)/2$, and it is a symmetric hyperbolic system.

To control the norm of gradients, our proof will intensively invoke the standard Hodge identity for vector fields $v\in H^1(\mathbb T^d;\mathbb R^d)$ ($d=2,3$) \cite{temam2001ns}:
\begin{equation} \label{eq:hodge}
    \|\nabla v\|_{L^2(\mathbb T^d)}^2 = \|\dive v\|_{L^2(\mathbb T^d)}^2 + \|\curl v\|_{L^2(\mathbb T^d)}^2.
\end{equation}
In particular, the identity applies to any vector field obtained by taking spatial derivatives of $v$.

For the vorticity $\omega=\curl u$ of (\ref{eq:ns}) and $\omega^{\epsilon,\delta}=\curl u^{\epsilon,\delta}$ of (\ref{eq:relax}), we shall frequently make use of the Sobolev embedding $H^1(\mathbb T^3)\hookrightarrow L^6(\mathbb T^3)$, which gives 
\begin{equation} \label{eq:sob_emb}
    \|\phi\|_{L^6(\mathbb T^3)}\le C\|\nabla\phi\|_{L^2(\mathbb T^3)} \quad \text{for  } \phi=\omega^{\epsilon,\delta}-\omega, \text{ or } \ \omega,
\end{equation} 
since $\int_{\mathbb T^3} \phi\,dx=0$ in both cases. Here, the constant $C>0$ depends only on the domain $\mathbb T^3$.

Finally, we need a statement on the blow-up of solutions of an  inhomogeneous system of transport equations. 

\begin{lemma} [Lemma 2.8 of \cite{HRYZ25}] \label{lem:blowup}
For $m \in \mathbb N$ and $T>0$, let $W: [0,T)\times {\mathbb T^d}  \to \mathbb R^m$ with $W\in C^0([0,T);H^2(\mathbb T^d))\cap C^1([0,T);H^1(\mathbb T^d))$ be a classical solution of the initial value problem for 
\[ \partial_t U + \sum^d_{i=1} A_i \partial_{x_i}U = Q(U) + F \text{ in $[0,T)\times \mathbb T^d$.}  \]
Here $F\in L^1_{loc}(\mathbb R_+;H^2(\mathbb T^d))\cap C^0(\mathbb R_+;H^1(\mathbb T^d))$, $A_i=A^\top_i\in \mathbb R^{m\times m}$, and $Q=Q(U)$ being a nontrivial quadratic function of $U$ with $Q(0)=0$ are given. In particular, the cases $Q\equiv 0$ or linear are excluded.
If the maximal existing time $T$ is finite, then we have
    \[ \int^T_0 \|W(t,\cdot)\|_{L^\infty(\mathbb T^d)}\, dt=\infty. \]
\end{lemma}
The original lemma in \cite{HRYZ25} is stated for $d=2$, but the proof therein works equally well for $d=3$.

\section{Vanishing amplitude of initial velocity perturbations on $\mathbb T^3$} \label{sec:small} 
This section is devoted to the proofs of Theorem~\ref{thm:uconverg_H1}, Corollary~\ref{cor:convergence}, and Theorem~\ref{thm:pressure_3d} in the regime (\ref{eq:small_Linf}). All involved norms are for $\mathbb T^3$. The linear symmetric hyperbolic structure of (\ref{eq:res}) is crucially exploited with the classical energy estimates. Throughout this section, the generic positive constant $C$ depends only on $C_0$ and the smooth solution $(p,u,U)$, and may take different values from line to line.

\subsection{Proof of Theorem~\ref{thm:uconverg_H1}} \label{subsec:thm:uconverg_H1}

Consider the system (\ref{eq:res}) for the residual $W=(q,v,V)$ in (\ref{eq:res_def}). Define an energy function as
\[
    E(t) := {\| W(t,\cdot)\|}_{H^1}^2 + (\epsilon+\delta)^a {\| \nabla^2 W(t,\cdot)\|}_{L^2}^2.
\]
It can be shown from the theorem assumption (\ref{eq:uH1_assump}) that the initial energy is bounded as
\[ E(0) \le C(\epsilon+\delta) \]
for $\epsilon+\delta\ll 1$. 

We then exploit the linear structure of the first-order operator in \eqref{eq:res} and the growth condition for the relaxation terms to find the estimate
\begin{equation} \label{eq:energy_est_1}
    \frac{d}{dt}E(t) \le C\left( \sqrt{(\epsilon+\delta)E(t)} + E(t) + (\epsilon+\delta)^{-\frac{a}{2}}E(t)^2 \right),
\end{equation}
when the solutions to both (\ref{eq:ns}) and (\ref{eq:relax}) exist. We give details to derive (\ref{eq:energy_est_1}) for the reader's reference. Taking for $\ell\in \{0,1,2\}$, the $\ell$-th spatial derivative of the equations in (\ref{eq:res}), multiplying each equation by $\nabla^\ell q$, $\nabla^\ell v$, and $\nabla^\ell V$, respectively, and integrating over $\mathbb T^3$, we obtain
\begin{equation} \label{eq:perform_energy_est}
\begin{aligned}
        \frac{1}{2}\frac{d}{dt}{\| \nabla^\ell W\|}_{L^2}^2 &= \int_{\mathbb T^3} \bigg[
        -\sqrt\epsilon (\nabla^\ell q) (\partial_t\nabla^\ell p) - \sqrt\delta (\nabla^\ell V):(\partial_t\nabla^\ell U) \\
       & \qquad \qquad \qquad + \frac{\nabla^\ell V}{\sqrt\delta}:\left( \nabla^\ell (2uv+v\otimes v) - \frac{\nabla^\ell V}{\sqrt\delta} \right) 
        \bigg] \, dx \\
        &\le C\left( \sqrt\epsilon {\| \nabla^\ell q\|}_{L^2} + \sqrt\delta {\|\nabla^\ell V\|}_{L^2} \right) + {\|\nabla^\ell (2uv+v\otimes v)\|}_{L^2}^2.
\end{aligned}
\end{equation}
Here, the estimate \eqref{eq:perform_energy_est} results from H{\"o}lder's inequality and the elementary fact 
\[
  (\nabla^\ell V/\sqrt\delta):[\nabla^\ell(2uv+v\otimes v)]\le |\nabla^\ell V|^2/\delta + |\nabla^\ell(2uv+v\otimes v)|^2.
\]
It hence follows that
\[
\begin{aligned}
    \frac{d}{dt}E(t) \le \, & C \sqrt\epsilon\left( \|q\|_{H^1} + (\epsilon+\delta)^a {\| \nabla^2 q\|}_{L^2} \right) + C\sqrt\delta\left( {\|V\|}_{H^1} + (\epsilon+\delta)^a \|\nabla^2 V\|_{L^2} \right) \\ 
    &+ C\left( \|v\|_{H^1}^2 + (\epsilon+\delta)^a \|v\|_{H^2}^2 \right) + A \\
    \le \, & C \left( \sqrt{(\epsilon+\delta)E(t)} + E(t) \right) + A
\end{aligned}
\]
for $\epsilon+\delta\ll 1$. Here $A=A(t):=C\sum_{\ell=0}^2 A_\ell(t)$, with
\begin{equation} \label{eq:As}
    A_0(t) = \|v\|_{L^4}^4, \quad A_1(t)= \|v\|_{L^\infty}^2 \|\nabla v\|_{L^2}^2, \quad
    A_2(t) = (\epsilon+\delta)^a \|v\|_{L^\infty}^2 \|\nabla^2 v\|_{L^2}^2
\end{equation}
serving as upper bounds for ${\|v\otimes v\|}_{L^2}^2$, ${\| \nabla (v\otimes v)\|}_{L^2}^2$, and $(\epsilon+\delta)^a {\| \nabla^2 (v\otimes v)\|}_{L^2}^2$, respectively, thanks to Lemma~\ref{lem:interp_ineq}. One can further deduce, term by term, that
\[ A_\ell(t) \le C (\epsilon+\delta)^{-a/2} E(t)^2, \quad \ell \in \{0,1,2\} \]
holds for $\epsilon+\delta\ll 1$. This can be verified by using Lemma~\ref{lem:interp_ineq} and the elementary facts that $E(t)\ge \|v\|_{H^1}^2$ and $(\epsilon+\delta)^{-a/2}E(t)\ge 2\|v\|_{H^1} {\| \nabla^2 v\|}_{L^2}$. The estimate (\ref{eq:energy_est_1}) is thus established.

We hence deduce from (\ref{eq:energy_est_1}) that, for either $t<T_{\epsilon,\delta}\le T_0$ or $t\le T_0<T_{\epsilon,\delta}$, 
\begin{equation} \label{eq:timebound}
    t \ge C \int_{C(\epsilon+\delta)}^{E(t)} \frac{dy}{\sqrt{(\epsilon+\delta)y}+y+(\epsilon+\delta)^{-\frac{a}{2}}y^2}
    \ge C \int_{C(\epsilon+\delta)}^{\min\{E(t),C(\epsilon+\delta)^{\frac{a}{2}}\}} \frac{dy}{y},
\end{equation}
since $\sqrt{(\epsilon+\delta)y}\le C^{-1/2} y$ and $(\epsilon+\delta)^{-a/2}y^2 \le Cy$ for $y\in[C(\epsilon+\delta), C(\epsilon+\delta)^{a/2}]$. As $t\to T_{\epsilon,\delta}$, it holds that $E(t)\to\infty$ (due to the blow-up criterion from Lemma \ref{lem:blowup}) and hence
\[
    T_{\epsilon,\delta}\ge C'(2-a)|\log(\epsilon+\delta)|,
\]
with $C'$ depending only on $C_0$ and the smooth solution $(p,u,U)$. This verifies the lower bound of $T_{\epsilon,\delta}$ in (\ref{eq:Texist}). 
    
On the other hand, for any $T>0$ such that both $u^{\epsilon,\delta}$ and $u$ exist on $[0,T]$, if $\epsilon+\delta$ is so small that
\[
    C'(2-a)|\log (\epsilon+\delta)| > T,
\]
then it is only possible from (\ref{eq:timebound}) that $E(t)\le C(\epsilon+\delta) e^{t/C}$ for $t\le T$. This amounts to
\begin{equation} \label{eq:E_bound}
    \sup_{t\in[0,T]} E(t) \le C_T (\epsilon+\delta)
\end{equation}
with $C_T$ depending only on $C_0$, $T$ and $(p,u,U)$, which immediately leads to the bound (\ref{eq:u_converg}) for $\|u^{\epsilon,\delta}-u\|_{H^1}$. 
    
Finally, the uniform boundedness (\ref{eq:u_Linf_bound}) of ${\| u^{\epsilon,\delta}\|}_{L^\infty}$ follows from (\ref{eq:E_bound}) and (\ref{eq:interp_inf})
\[
    \|v(t,\cdot)\|_{L^\infty}^2 \le C\|v\|_{H^1} \|v\|_{H^2} \le C \|v\|_{H^1}^2 + C\|v\|_{H^1}\|\nabla^2 v\|_{L^2} \le CE+C(\epsilon+\delta)^{-\frac{a}{2}}E \le C_{T}
\]
for all $t\le T$ and $(\epsilon,\delta)$, due to $a<2$. This completes the proof.

\subsection{Proof of Corollary~\ref{cor:convergence}} \label{subsec:cor:convergence}

Define the same residual $W=(q,v,V)$ as in (\ref{eq:res_def}), which is governed by (\ref{eq:res}). Take an energy as
\[
    E(t):=\|W(t,\cdot)\|_{L^2}^2 + (\epsilon+\delta)^a \|\nabla^2 W(t,\cdot)\|_{L^2}^2,
\]
and hence $E(0)\le C(\epsilon+\delta)$. The energy estimate gives, for $t\in[0,T_0]$,
\begin{equation} \label{eq:energy_est_2}
    \frac{d}{dt}E(t) \le C\left( \sqrt{(\epsilon+\delta)E(t)} + E(t) + (\epsilon+\delta)^{-\frac{3a}{4}}E(t)^2 \right).
\end{equation}
Similar to the proof of Theorem~\ref{thm:uconverg_H1}, the key to derive (\ref{eq:energy_est_2}) is to show that
\[ A_\ell(t) \le C(\epsilon+\delta)^{-3a/4}E^2, \quad \ell \in \{0,2\} \]
(noticing that the first-order derivatives are not included in $E$), where the $A_\ell$'s are defined in (\ref{eq:As}). This can be done by using Lemma~\ref{lem:interp_ineq} and $E\ge \frac{1}{4}\|v\|_{L^2}^2+\frac{3}{4}(\epsilon+\delta)^a\|\nabla^2v\|_{L^2}^2\ge (\epsilon+\delta)^{3a/4}\|v\|_{L^2}^{1/2}\|\nabla^2v\|_{L^2}^{3/2}$.

With (\ref{eq:energy_est_2}), we set $a'=3a/2\in[0,2)$, and the remaining steps are exactly the same as in the proof of Theorem~\ref{thm:uconverg_H1}.

\subsection{Proof of Theorem~\ref{thm:pressure_3d}} \label{subsec:thm:pressure_3d}

Consider $T>0$ such that both $(p^{\epsilon,\delta},u^{\epsilon,\delta},U^{\epsilon,\delta})$ and $(p,u,U)$ exist on $[0,T]$. Following the procedure in the proof of Theorem 2.2 in \cite{HRYZ25}, we consider an intermediate affine  system of transport equations for the unknowns $(p',u',U')$ given by 
\begin{equation} \label{eq:intermediate}
\left \{ \
\begin{aligned}
  &\epsilon \partial_t p' + \nabla \cdot u' =0, \\
  &\partial_t u' + \nabla \cdot U' + \nabla p' = 0, \\
  &\delta \partial_t U' + \nabla u' = u\otimes u - U', \\
  &u'(0,\cdot)=u^{\epsilon,\delta}_0, \ p'(0,\cdot)=p^{\epsilon,\delta}_0, \ U'(0,\cdot)=U^{\epsilon,\delta}_0.
\end{aligned}
\right.
\end{equation}
The key strategy is to estimate ${\|p^{\epsilon,\delta}-p'\|}_{H^1}$ and ${\|p'-p\|}_{H^1}$ separately so as to control ${\|p^{\epsilon,\delta}-p\|}_{H^1}$, and to follow the same approach for ${\|u^{\epsilon,\delta}-u\|}_{H^1}$.
In particular, in \textbf{Step I} below, we will show
\begin{equation} \label{eq:pprime_p_h1}
    \sup_{t\in[0,T]}  \sqrt\epsilon\|(p'-p)(t,\cdot)\|_{H^1}\le C_T(\epsilon+\delta),
\end{equation}
followed by \textbf{Step II} validating 
\begin{equation} \label{eq:uprime_u_h1}
    \sup_{t\in[0,T]} \|(u'-u)(t,\cdot)\|_{H^1}\le C_T(\epsilon+\delta),
\end{equation}
and finally in \textbf{Step III}
\begin{equation} \label{eq:relax_interm}
    \sup_{t\in[0,T]} \|(u^{\epsilon,\delta}-u')(t,\cdot)\|_{H^1} +  \sqrt\epsilon\|(p^{\epsilon,\delta}-p')(t,\cdot)\|_{H^1}\le C_T(\epsilon+\delta)
\end{equation}
for $\epsilon+\delta\ll 1$, where the constant $C_T$ depends only on $C_0$, $a$, $T$, and the smooth solution $(p,u,U)$.

Before we proceed with \textbf{Steps I},{\bf II}, {\bf III}, notice that (\ref{eq:intermediate}), as an inhomogeneous system with smooth initial data and right-hand-side with linear growth, admits a solution on $[0,T]$. Analogous to (\ref{eq:perform_energy_est}), it is straightforward to get
\[
    \frac{1}{2}\frac{d}{dt} \left\|\nabla^2 \left( \sqrt\epsilon p', u', \sqrt\delta U' \right) \right\|_{L^2}^2
    = \int_{\mathbb T^3} \nabla^2U':\nabla^2(u\otimes u-U') \, dx
    \le \frac{1}{4} \|u\otimes u\|_{H^2}^2 \le C
\]
for $t\in [0,T]$. By the initial condition (\ref{eq:3dp_assump_3}), we have
\begin{equation} \label{eq:uprime_h2}
    \left\|\Delta \left( \sqrt\epsilon p', u', \sqrt\delta U' \right)(t,\cdot) \right\|_{L^2}^2 \le C(\epsilon+\delta)^{-2a} + CT_0 \le C(\epsilon+\delta)^{-2a}
\end{equation}
for $t\in [0,T]$ and $\epsilon+\delta\ll 1$. This estimate will be used later.

The remainder of the proof consists of the three steps.

\textbf{Step I}. 
By rearranging (\ref{eq:intermediate}) we see that $p'$ is governed by
\begin{equation} \label{eq:plinear}
    \epsilon\delta\partial_t^3 p' - (\epsilon+\delta)\partial_t \Delta p' + \epsilon \partial_t^2 p' - \Delta (p'-p) = 0,
\end{equation}
which is derived in the same way as in  Proposition 4.2 of \cite{HRYZ25}. Introducing
\begin{subequations}
\begin{align}
    f &= \delta \partial_t p' + p'-p, \label{eq:fdef} \\
    g &=\sqrt{\epsilon\delta}\partial_t\nabla p',
\end{align}
\end{subequations}
we obtain
\begin{subequations} \label{eq:fg}
    \begin{align}
        (\epsilon\partial_t^2-\Delta)f - \sqrt{\frac{\epsilon}{\delta}} \nabla\cdot g &= \epsilon \partial_t^2 p, \label{eq:fga} \\
        \partial_t g - \sqrt{\frac{\epsilon}{\delta}} \nabla \partial_t f &= -\frac{g}{\delta} + \sqrt{\frac{\epsilon}{\delta}} \partial_t \nabla p, \label{eq:fgb}
    \end{align}
\end{subequations}
subject to
\[
\begin{aligned}
    f(0,\cdot) &= -\frac{\delta}{\epsilon}\dive u'(0,\cdot) + p'(0,\cdot)-p(0,\cdot), \\
    \partial_t f(0,\cdot) &= \frac{\delta}{\epsilon}\left( \dive^2 U'(0,\cdot) + \Delta p'(0,\cdot) \right) - \frac{1}{\epsilon}\dive u'(0,\cdot) - \partial_tp(0,\cdot), \\
    g(0,\cdot)&= -\sqrt{\frac{\delta}{\epsilon}} \nabla\dive u'(0,\cdot).
\end{aligned}
\]
Here, (\ref{eq:fga}) is the reformulation of (\ref{eq:plinear}), and (\ref{eq:fgb}) is the restriction given by the definition of $f$ and $g$. The system (\ref{eq:fg}) is endowed with an energy
\[
    E(t)=\int_{\mathbb T^3} \left( \epsilon |\partial_t f|^2 + |\nabla f|^2 + |g|^2 \right)\, dx
\]
with
\[ \epsilon E(0) \le C(\epsilon+\delta)^2 \]
(see (4.16) in \cite{HRYZ25} for a derivation of the bound for $E(0)$ in 2D, which works equally well for 3D). Then, the energy estimate reads as
\[
\begin{aligned}
    \frac{d}{dt}E(t) &\le 2\epsilon\int_{\mathbb T^3} (\partial_t f) (\partial_t^2p) \, dx + 2\int_{\mathbb T^3} \left( \frac{g}{\sqrt\delta} \sqrt\epsilon \partial_t\nabla p - 2\frac{|g|^2}{\delta} \right) \, dx\\
    &\le C \epsilon \|\partial_t f\|_{L^2} + C\epsilon \\
    &\le C\left( \sqrt{\epsilon E(t)} + \epsilon \right),
\end{aligned}
\]
implying that
\[
    \epsilon E(t)\le \epsilon E(0) + C \epsilon^2(1+t^2)
\]
for $t\in[0,T]$ and hence
\begin{equation} \label{eq:fg_energy_est}
    \sup_{t\in[0,T]} \epsilon E(t) \le C_T(\epsilon+\delta)^2 
\end{equation}
by using the bound for $E(0)$.

Now we can get a refined estimate of $\nabla(p'-p)$. We view (\ref{eq:fdef}) as a first-order ODE
\[
    \partial_t \mathcal{A} p' + \frac{1}{\delta} \mathcal{A}p' = \frac{1}{\delta} \mathcal{A}f + \frac{1}{\delta} \mathcal{A}p
\]
for $\mathcal{A} \in \{\nabla,\, \partial_t\}$, and the solution gives
\begin{equation} \label{eq:Apprime_p}
\begin{split}
    \mathcal{A}(p'-p)(t) =\, & e^{-\frac{t}{\delta}} \left[	\mathcal{A}(p'-p)(0) + \mathcal{A}(p(0)-p(t)) \right] \\
    &+ \int^t_0 \frac{1}{\delta} e^{\frac{t'-t}{\delta}} \mathcal{A}f(t')\,dt'
    + \int^t_0 \frac{1}{\delta} e^{\frac{t'-t}{\delta}} \mathcal{A}(p(t')-p(t))\,dt'.
\end{split}
\end{equation}

For $\mathcal{A}=\nabla$, we have 
\[
\begin{aligned}
    \sup_{t\in[0,T]} \sqrt\epsilon{ \| \nabla(p'-p)(t,\cdot)\|}_{L^2} &\le \sqrt\epsilon{\| \nabla (p'-p)(0,\cdot)\|}_{L^2} + \sup_{t\in[0,T]} \sqrt{\epsilon E(t)} + C_T\sqrt\epsilon\delta \\
    &\le C_T(\epsilon+\delta).
\end{aligned}
\]
Here, the first line is derived in the same way as (4.14) in \cite{HRYZ25}; note that the term $\sup_{t\in[0,T]} e^{-\frac{t}{\delta}}\|\nabla p(0) - \nabla p(t) \|_{L^2}$ is bounded by $C\sup_{t\in[0,T_0]} te^{-\frac{t}{\delta}} = C\delta$. The second line follows directly from (\ref{eq:3dp_assump_1}) and (\ref{eq:fg_energy_est}). This suffices to ensure (\ref{eq:pprime_p_h1}) by the facts that $\int_{\mathbb T^3}p\,dx=0$, $\frac{d}{dt}\int_{\mathbb T^3} p'\,dx=0$, and the Poincar{\'e} inequality:
\[
\begin{aligned}
    \sqrt\epsilon\|(p'-p)(t,\cdot)\|_{L^2} &\le C\sqrt\epsilon\left| \int_{\mathbb T^3} (p'-p)(0,x)\,dx \right| + C\sqrt\epsilon{\|\nabla(p'-p)(t,\cdot)\|}_{L^2} \\
    &\le C\sqrt\epsilon{\|(p'-p)(0,\cdot)\|}_{L^2} + C\sqrt\epsilon{\|\nabla(p'-p)(t,\cdot)\|}_{L^2} \\
    &\le C(\epsilon+\delta).
\end{aligned}
\]

\textbf{Step II}. Now we estimate ${\|u'-u\|}_{H^1} \le {\|u'-u\|}_{L^2}+{\|\nabla(u'-u)\|}_{L^2}$. Noting that
\[
    \partial_t\int_{\mathbb T^3} u'(t,x)\,dx=\partial_t\int_{\mathbb T^3} u(t,x)\,dx=0,
\]
and using the Poincar\'e inequality, we obtain
\[
\begin{aligned}
    {\|u'-u\|}_{L^2} &\le \left | \int_{\mathbb T^3} (u'-u)(0,x)\,dx \right| + {\|\nabla(u'-u)\|}_{L^2} \\
    &\le {\|(u'-u)(0,\cdot)\|}_{L^2} + {\|\nabla(u'-u)\|}_{L^2}.
\end{aligned}
\]
Thanks to (\ref{eq:3dp_assump_1}), it suffices to prove
\begin{equation}
    \sup_{t\in[0,T]}{\| \nabla(u'-u)(t,\cdot)\|}_{L^2} \le C_T(\epsilon+\delta).
\end{equation}
for a validation of (\ref{eq:uprime_u_h1}). For this purpose, we use (\ref{eq:hodge}) as $\|\nabla(u'-u)\|_{L^2}^2 = \|\dive u'\|_{L^2}^2 + \|\curl(u'-u)\|_{L^2}^2$ (noting that $\dive u=0$) and prove 
\begin{equation} \label{eq:divuprime}
    \sup_{t\in[0,T]} {\| \dive u'(t,\cdot)\|}_{L^2} \le C_T(\epsilon+\delta)
\end{equation}
and
\begin{equation} \label{eq:curluprime_u_L2}
    \sup_{t\in[0,T]}{ \| \curl(u'-u)(t,\cdot)\|}_{L^2} \le C_T(\epsilon+\delta),
\end{equation}
respectively.

First, (\ref{eq:divuprime}) can be derived by taking $\mathcal{A} = \partial_t$ in (\ref{eq:Apprime_p}) and using $\partial_t p' = -\epsilon^{-1} \dive u'$.

For (\ref{eq:curluprime_u_L2}), denote $\omega' = \curl u' \in\mathbb R^3$ and $\Omega'=\curl U'\in \mathbb R^{3\times 3}$. Here the $j$th column of $\Omega'$ is defined as $\Omega'_{\cdot,j}=\curl U'_{\cdot,j}$ for $j=1,2,3$.
The governing equations for $\omega'$ and $\Omega'$ read as
\[
\left\{ 
\begin{aligned}
    &\partial_t \omega' + \nabla\cdot\Omega'=0, \\
    &\delta \partial_t \Omega' + \nabla \omega' = \curl(u\otimes u)-\Omega'.
\end{aligned}
\right.
\]
Correspondingly, denote by $\omega = \curl u$ and $\Omega=\curl U$ the vorticity of (\ref{eq:ns}), and these variables are governed by
\[
\left\{ 
\begin{aligned}
    &\partial_t \omega + \nabla\cdot\Omega=0, \\
    &\Omega= - \nabla \omega + \curl(u\otimes u).
\end{aligned}
\right.
\]
Denoting
\[
    \xi=\frac{\omega'-\omega}{\sqrt\delta}, \quad
    X = \Omega'-\Omega - e^{-\frac{t}{\delta}}\left( \Omega'(0,\cdot)-\Omega(0,\cdot) \right),
\]
we have
\begin{equation}
\left\{ 
\begin{aligned}
    & \partial_t\xi + \frac{1}{\sqrt\delta} \nabla\cdot X = -\frac{1}{\sqrt\delta} e^{-\frac{t}{\delta}} \nabla\cdot \left( \Omega'(0,\cdot)-\Omega(0,\cdot) \right), \\
    & \partial_t X + \frac{1}{\sqrt\delta}\nabla \xi = -\frac{X}{\delta} - \partial_t\Omega
\end{aligned}
\right.
\end{equation}
subject to $X(0,\cdot)=0$ and $\sqrt\delta\|\xi(0,\cdot)\|_{L^2} = {\| (\omega'-\omega)(0,\cdot)\|}_{L^2} \le C(\epsilon+\delta)$ by theorem assumption (\ref{eq:3dp_assump_1}).

The energy estimate is
\[
\begin{aligned}
    \frac{d}{dt}{\| (\xi,X)\|}_{L^2}^2 &\le \frac{C}{\sqrt\delta} e^{-\frac{t}{\delta}} {\| \nabla\cdot \left( \Omega'(0,\cdot)-\Omega(0,\cdot) \right)\|}_{L^2}\|\xi\|_{L^2} + C\delta \\
    &\le \frac{C(\epsilon+\delta)}{\delta^{\frac{3}{2}}} e^{-\frac{t}{\delta}} \|\xi\|_{L^2} + C\delta,
\end{aligned}
\]
where we used $\| \dive(\Omega'-\Omega)(0,\cdot)\|_{L^2}\le \|\nabla \dive (U'-U)(0,\cdot)\|_{L^2}\le C(\epsilon+\delta)$. To proceed, one may further define
\[
    h=\left( \|(\xi,X)\|_{L^2}^2 + \delta \right)^{\frac{1}{2}},
\]
and the above energy estimate implies that
\[
\begin{aligned}
    \frac{dh}{dt} &\le Ch + \frac{C(\epsilon+\delta)}{\delta^{\frac{3}{2}}}e^{-\frac{t}{\delta}}, \\
    h(0) &\le \frac{C(\epsilon+\delta)}{\sqrt\delta},
\end{aligned}
\]
which can be solved as
\[
    h(t) \le \frac{C(\epsilon+\delta)}{\sqrt\delta}e^{Ct}\left(1+\frac{1}{C\delta+1}\left( 1-e^{-\left(C+\frac{1}{\delta}\right)t} \right)\right) \le \frac{C_{T_0}(\epsilon+\delta)}{\sqrt\delta}
\]
for $t\in[0,T]$. Therefore, (\ref{eq:curluprime_u_L2}), and hence (\ref{eq:uprime_u_h1}), follow immediately.

Moreover, using (\ref{eq:interp_inf}), we see that
\begin{equation} \label{eq:uprime_u_Linf}
    \|u'-u\|_{L^\infty}^2\le C\|u'-u\|_{H^1}\|u'-u\|_{H^2}
    \le C(\epsilon+\delta) + C(\epsilon+\delta)^{1-a} \le C,
\end{equation}
where the second inequality results from (\ref{eq:uprime_u_h1}) and (\ref{eq:uprime_h2}). This immediately leads to
\begin{equation} \label{eq:uprime_Linf}
    \sup_{t\in[0,T]} \|u'(t,\cdot)\|_{L^\infty} \le C_T.
\end{equation}
These estimates will be needed in the sequel.

\textbf{Step III}. Finally, we estimate the difference between (\ref{eq:intermediate}) and (\ref{eq:relax}). For this purpose, we denote
\[
    v = u^{\epsilon,\delta}-u', \quad
    q = \sqrt{\epsilon}\left(p^{\epsilon,\delta}-p'\right), \quad
    V = \sqrt{\delta}\left(U^{\epsilon,\delta}-U'\right).
\]
Then, the system for $W=(q,v,V)$ reads as
\[
\left\{
\begin{aligned}
    &\partial_t q + \frac{1}{\sqrt{\epsilon}}\nabla\cdot v=0, \\
    &\partial_t v + \frac{1}{\sqrt{\epsilon}}\nabla q + \frac{1}{\sqrt{\delta}}\nabla\cdot V = 0, \\
    &\partial_t V + \frac{1}{\sqrt{\delta}}\nabla v = \frac{v\otimes v+2u'v}{\sqrt\delta} + \frac{u'-u}{\sqrt\delta}\otimes u' + u\otimes \frac{u'-u}{\sqrt\delta}-\frac{V}{\delta},
\end{aligned}
\right.
\]
with $u'v:=(u'\otimes v+v\otimes u')/2$, subject to initial datum $W(0,\cdot)=0$.

Define the energy as
\[
    E(t) = \|W(t,\cdot)\|_{H^1}^2 + (\epsilon+\delta)^{2a+2} \|\nabla^2 W(t,\cdot)\|_{L^2}^2,
\]
and the initial conditions (\ref{eq:3dp_assump_1})-(\ref{eq:3dp_assump_3}) imply that $E(0)\le C(\epsilon+\delta)^2$. An energy estimate gives
\begin{equation} \label{eq:energy_est_3}
    \frac{d}{dt}E(t) \le C\left(\|v\|_{L^\infty}^2+1\right)E(t) + C\left(\epsilon+\delta\right)^2(\|v\|_{L^\infty}^2+1).
\end{equation}
Indeed, analogous to (\ref{eq:perform_energy_est}), we have
\[ 
\begin{aligned}
    \frac{1}{2}\frac{d}{dt}\|\nabla^\ell W\|_{L^2}^2 &\le  C\left(\|\nabla^\ell (v\otimes v)\|_{L^2}^2 + \|\nabla^\ell (u'v)\|_{L^2}^2 + \|\nabla^\ell \left((u'-u)\otimes u'+u\otimes (u'-u)\right)\|_{L^2}^2 \right) \\
    &=:C\left( I^\ell_1 + I^\ell_2 + I^\ell_3 \right) 
\end{aligned}
\]
for $\ell=0,1,2$ and that
\[
\begin{aligned}
    I^0_1 &\le C\|v\|_{L^\infty}^2\|v\|_{L^2}^2 \le C\|v\|_{L^\infty}^2E(t), \quad
    I^0_2 \le C\|v\|_{L^2}^2 \le CE(t), \\
    I^0_3 &\le  \|u'-u\|_{L^2}^2 \left( \|u'\|_{L^\infty}^2+\|u\|_{L^\infty}^2 \right) \le C \|u'-u\|_{L^2}^2 \le C(\epsilon+\delta)^2, \\
    I^1_1 &\le C \|v\|_{L^\infty}^2 \|\nabla v\|_{L^2}^2 \le C\|v\|_{L^\infty}^2 E(t), \\
    I^1_2 & \le \| \nabla \left((u'-u) v \right)\|_{L^2}^2 + \|\nabla(uv)\|_{L^2}^2 \\ 
    &\le C \left( \|v\|_{L^\infty}^2 \|\nabla(u'-u)\|_{L^2}^2 + \|u'-u\|_{L^\infty}^2 \|\nabla v\|_{L^2}^2 + \|v\|_{H^1}^2 \right) \\
    &\le C\|v\|_{L^\infty}^2 (\epsilon+\delta)^2 + CE, \\
    I^1_3 &\le C\left( {\| \nabla\left( (u'-u)\otimes (u'-u) \right)\|}_{L^2}^2 + {\| \nabla\left( (u'-u)\otimes u \right)\|}_{L^2}^2 \right) \\
    & \le C\left(\|u'-u\|_{L^\infty}^2\|u'-u\|_{H^1}^2 + \|u'-u\|_{H^1}^2\right) \le C(\epsilon+\delta)^2, \\
    (\epsilon+\delta)^{2a+2}I^2_1 &\le  C(\epsilon+\delta)^{2a+2} \|v\|_{L^\infty}^2\|\nabla^2 v\|_{L^2}^2 \le C\|v\|_{L^\infty}^2 E(t), \\
    (\epsilon+\delta)^{2a+2}I^2_2 &\le C(\epsilon+\delta)^{2a+2} \left( \|u'\|_{L^\infty}^2 \|\nabla^2v\|_{L^2}^2 + \|\nabla v\|_{L^2}^2 + \|v\|_{L^2}^2 + \|v\|_{L^\infty}^2 {\| \nabla^2(u'-u)\|}_{L^2}^2 \right) \\
    &\le CE(t) + C\|v\|_{L^\infty}^2 (\epsilon+\delta)^{2a+2}\left( \|\nabla^2 u'\|_{L^2}^2 + \|\nabla^2 u\|_{L^2}^2 \right) \\
    &\le CE(t) + C\|v\|_{L^\infty}^2 (\epsilon+\delta)^2,
\end{aligned}
\]
and
\[
\begin{aligned}
    (\epsilon+\delta)^{2a+2}I^2_3 &\le C (\epsilon+\delta)^{2a+2} \left( {\| \nabla^2\left( (u'-u)\otimes (u'-u) \right)\|}_{L^2}^2 + {\| \nabla^2\left( (u'-u)\otimes u \right)\|}_{L^2}^2 \right) \\
    &\le C (\epsilon+\delta)^{2a+2} \left({\|u'-u\|_{L^\infty}^2 \| \nabla^2(u'-u)\|}_{L^2}^2 + \|u'-u\|_{H^2}^2 \right) \\
    & \le C(\epsilon+\delta)^2
\end{aligned}
\]
for $\epsilon+\delta\ll 1$, where (\ref{eq:uprime_u_h1}), (\ref{eq:uprime_u_Linf}) and (\ref{eq:uprime_Linf}) were used throughout; moreover, in the estimate of $I^2_2$ and $I^2_3$, we used (\ref{eq:uprime_h2}). The estimate (\ref{eq:energy_est_3}) thus follows.

Further noticing from (\ref{eq:interp_inf}) that
\[ {\|v\|}_{L^\infty}^2 \le C {\|v\|}_{H^1} {\|v\|}_{H^2} \le C(\epsilon+\delta)^{-a-1}E(t), \]
we derive from (\ref{eq:energy_est_3}) that
\[
    \frac{d}{dt}E(t) \le C\left( (\epsilon+\delta)^2 + E(t) + (\epsilon+\delta)^{-a-1}E(t)^2 \right),
\]
and hence
\[
    t \ge C\int_{C(\epsilon+\delta)^2}^{E(t)} \frac{dy}{(\epsilon+\delta)^2+y+(\epsilon+\delta)^{-a-1}y^2} 
    \ge C\int_{C(\epsilon+\delta)^2}^{\min\{E(t),C(\epsilon+\delta)^{a+1}\}} \frac{dy}{y}.
\]
Then the same argument as in the proof of Theorem~\ref{thm:uconverg_H1} yields that the solution $(u^{\epsilon,\delta},p^{\epsilon,\delta},U^{\epsilon,\delta})$ of (\ref{eq:relax}) exists on $[0,T]$. Moreover, for sufficiently small $\epsilon+\delta\ll 1$, it is only possible that
\[
    \sup_{t\in[0,T]} E(t) \le C_T(\epsilon+\delta)^2,
\]
which immediately gives (\ref{eq:relax_interm}). This completes the proof of Theorem~\ref{thm:pressure_3d} under $\delta=\mathcal O(\sqrt\epsilon)$. 

\section{Order-one initial velocity perturbations} \label{sec:large}

This section is devoted to the proofs of Theorems~\ref{thm:critical_3d}, \ref{thm:critical_3d_inhomo} and \ref{thm:critical_2d} in the more challenging regime (\ref{eq:small_Linf}). Now, a direct energy method applied to the rescaled symmetric hyperbolic system (\ref{eq:res}) is no longer sufficient to yield convergence in a ``long'' time interval. For this reason, we adopt the modulated energy method developed in \cite{Brenier2003}, which allows one compare the relaxed dynamics with the target incompressible flow at the correct level. Moreover, due to the presence of the artificial compressibility term, the velocity field is no longer divergence-free, and an additional smallness of $\dive u^{\epsilon,\delta}$ is required. Using a careful bootstrap argument, we are able to achieve global convergence results under the regime $\epsilon=o(\delta)$.

We shall start with some technical preparations related to the vorticity $\omega = \curl u$ for (\ref{eq:ns}) and $\omega^{\epsilon,\delta} = \curl u^{\epsilon,\delta}$ for (\ref{eq:relax}). Here we work on $\mathbb T^3$, so both $\omega$ and $\omega^{\epsilon,\delta}$ are vector-valued functions. Nevertheless, all the results here also holds on $\mathbb T^2$, which will be processed in Section~\ref{subsec:thm:critical_2d}.

With the $\curl$-operator acting on (\ref{eq:nsb}) and denoting 
\begin{equation} \label{eq:A0_def}
    A:=\curl\dive(u\otimes u) = u\cdot \nabla \omega - \omega \cdot \nabla u,
\end{equation}
we see that the components of $\omega$ are governed by the inhomogeneous heat equations 
\begin{equation} \label{eq:omega_ns}
    \partial_t\omega - \Delta \omega = -A.
\end{equation}
Now, applying the $\curl$-operator to (\ref{eq:relax_b}), we get by $\curl \dive U^{\epsilon,\delta} = -\partial_t \omega^{\epsilon,\delta}$ because of $\curl\nabla p^{\epsilon,\delta}=0$ the damped wave equations 
\begin{equation} \label{eq:omega_ed}
    \partial_t \left( \omega^{\epsilon,\delta} + \delta\partial_t\omega^{\epsilon,\delta} \right) - \Delta \omega^{\epsilon,\delta} =  A^{\epsilon,\delta}, \qquad A^{\epsilon,\delta}:= \curl\dive(u^{\epsilon,\delta}\otimes u^{\epsilon,\delta}).
\end{equation}
In (\ref{eq:omega_ed}), we used $\curl\Delta u^{\epsilon,\delta}=\Delta\curl u^{\epsilon,\delta}=\Delta \omega^{\epsilon,\delta}$.  It follows from a straightforward calculation that
\begin{equation} \label{eq:A_def}
    A^{\epsilon,\delta} = \curl(u^{\epsilon,\delta}\dive u^{\epsilon,\delta}) + u^{\epsilon,\delta}\cdot \nabla \omega^{\epsilon,\delta} - \omega^{\epsilon,\delta}\cdot \nabla u^{\epsilon,\delta} + \omega^{\epsilon,\delta}\dive u^{\epsilon,\delta}.
\end{equation}

For the system (\ref{eq:omega_ed}), we define the energy as
\begin{equation} \label{eq:Fbar}
    \bar F^{\epsilon,\delta} = \bar F^{\epsilon,\delta}(t) = \int_{\mathbb T^3} \left( \left|\omega^{\epsilon,\delta} + \delta\partial_t \omega^{\epsilon,\delta}\right|^2 + \left|\delta\partial_t \omega^{\epsilon,\delta}\right|^2 + 2\delta \left|\nabla \omega^{\epsilon,\delta} \right|^2 \right)\, dx,
\end{equation}
and the modulated energy, introduced in \cite{Brenier2003}, as 
\begin{equation} \label{eq:modulated_energy}
    F^{\epsilon,\delta} =    F^{\epsilon,\delta}(t) = \int_{\mathbb T^3} \left( \left|\omega^{\epsilon,\delta}-\omega + \delta\partial_t\omega^{\epsilon,\delta} \right|^2 +\left|\delta\partial_t\omega^{\epsilon,\delta}\right|^2 + 2\delta \left| \nabla \omega^{\epsilon,\delta} \right|^2 \right) \,dx.
\end{equation}
Let us prove the basic identities for $\bar F^{\epsilon,\delta}$ and $F^{\epsilon,\delta}$.

\begin{proposition}
The energy $\bar F^{\epsilon,\delta}$ and $F^{\epsilon,\delta}$ satisfy
\begin{equation} \label{eq:F0_evolve}
    \frac{d}{dt}\frac{1}{2}\bar F^{\epsilon,\delta}(t) = -\int_{\mathbb T^3} \left( \delta \left|\partial_t \omega^{\epsilon,\delta} \right|^2 + \left| \nabla\omega^{\epsilon,\delta} \right|^2 \right) \, dx -\int_{\mathbb T^3} \left( \omega^{\epsilon,\delta}+2\delta\partial_t \omega^{\epsilon,\delta} \right)\cdot A^{\epsilon,\delta} \, dx
\end{equation}
and
\begin{equation} \label{eq:modu_energy_evolve}
\begin{split}
    \frac{d}{dt}\frac{1}{2}F^{\epsilon,\delta}(t) &= T^{\epsilon,\delta}_\text{I}(t) + T^{\epsilon,\delta}_\text{II}(t) - \int_{\mathbb T^3}\left( \left| \nabla(\omega^{\epsilon,\delta}-\omega) \right|^2 + \delta \left| \partial_t\omega^{\epsilon,\delta} \right|^2 \right)\,dx, \\
    T^{\epsilon,\delta}_\text{I}(t) &:= \int_{\mathbb T^3} \delta\partial_t\omega^{\epsilon,\delta} \cdot\left(A-\Delta\omega-2A^{\epsilon,\delta} \right) \, dx, \\
    T^{\epsilon,\delta}_\text{II}(t) &:= - \int_{\mathbb T^3} (\omega^{\epsilon,\delta}-\omega)\cdot(A^{\epsilon,\delta}-A)\,dx,
\end{split}
\end{equation}
with  $A$ from \eqref{eq:A0_def} and $A^{\epsilon,\delta}$ from \eqref{eq:omega_ed}.
\end{proposition}

\begin{proof}
First, (\ref{eq:F0_evolve}) can be derived by multiplying (\ref{eq:omega_ed}) with $\omega^{\epsilon,\delta} + 2\delta \partial_t \omega^{\epsilon,\delta}$ and integrating (by parts) over $\mathbb T^3$.

Then, we note that
\[
    F^{\epsilon,\delta}(t) = \bar F^{\epsilon,\delta}(t) + \int_{\mathbb T^3} |\omega|^2\,dx - 2\int_{\mathbb T^3} \omega \cdot \left( \omega^{\epsilon,\delta} +\delta\partial_t \omega^{\epsilon,\delta} \right) \,dx.
\]
By using \eqref{eq:omega_ns} we have
\[
    \frac{d}{dt}\frac{1}{2}\int_{\mathbb T^3}|\omega|^2\,dx = -\int_{\mathbb T^3}|\nabla \omega|^2\,dx - \int_{\mathbb T^3}\omega\cdot A\, dx
\]
and hence
\[
\begin{aligned}
    \frac{d}{dt}\frac{1}{2}F^{\epsilon,\delta}(t) =& \frac{d}{dt}\frac{1}{2}\bar F^{\epsilon,\delta}(t) - \int_{\mathbb T^3}|\nabla \omega|^2\,dx - \int_{\mathbb T^3}\omega\cdot A\, dx \\
    &+ \int_{\mathbb T^3} \left( (A-\Delta\omega) \cdot \left(\omega^{\epsilon,\delta} + \delta \partial_t \omega^{\epsilon,\delta}\right) + \omega \cdot (A^{\epsilon,\delta}-\Delta \omega^{\epsilon,\delta}) \right) \,dx,
\end{aligned}
\]
where (\ref{eq:omega_ed}) and (\ref{eq:omega_ns}) were used. Then, (\ref{eq:modu_energy_evolve}) is derived by substituting (\ref{eq:F0_evolve}) into the above equation, integrating by parts, and rearranging the terms.
\end{proof}

\subsection{Proof of Theorem~\ref{thm:critical_3d} on $\mathbb T^3$} \label{subsec:thm:critical_3d}
The proof is built on a bootstrap argument. To treat the constants more carefully, we denote
\[
    K := 1+\|(p,u,U)\|_{C^0H^3}^2.
\]
In this proof we will use $C$ as a generic positive constant that depends only on $C_0$ in (\ref{ass1}) and (\ref{ass2}) (and the domain $\mathbb T^3$). This point will become clear later.

We introduce the bootstrap assumption that there exist constants $M, \mu>0$, both independent of $\delta$, such that, for sufficiently small $\delta\ll 1$, the following bounds hold for all $t\in[0,T_0]$:
\begin{subequations} \label{eq:critical_assum_1}
\begin{align}
    &{\|u^{\epsilon,\delta}(t,\cdot)\|}_{L^\infty}^2 + {\| \nabla u^{\epsilon,\delta}(t,\cdot)\|}_{L^3}^2 \le M, \label{assump_a} \\
    &\delta^{-\frac{1}{2}}\|\dive u^{\epsilon,\delta}(t,\cdot)\|_{L^2}^2 + \delta^{\frac{1}{2}} {\| \nabla\dive u^{\epsilon,\delta}(t,\cdot)\|}_{L^2}^2 \le \mu, \label{assump_b}\\
    &\|\left(u^{\epsilon,\delta}-u\right)(t,\cdot)\|_{L^2}^2 \le \mu \delta^{\frac{1}{2}}. \label{assump_c}
\end{align}
\end{subequations}
Moreover, we require that there exists a sufficiently large constant $C_1$, independent of $\delta$, such that
\begin{subequations} \label{eq:impose}
\begin{align}
    &M \ge C_1(\delta^{-\frac{1}{2}}F(0)+1) \quad \text{for all} \quad \delta\ll 1, \label{impose_a}\\
    &\mu \le \left[ (K+M)(1+T_0) \right]^{-1}. \label{impose_b}
\end{align}
\end{subequations}

We first show that the assumptions (\ref{eq:critical_assum_1}) and (\ref{eq:impose}) hold at $t=0$. By the initial assumptions (\ref{ass1}), (\ref{ass2}), it is clear that both $\|u^{\epsilon,\delta}_0-u_0\|_{L^\infty}^2$ and $\|\nabla(u^{\epsilon,\delta}_0-u_0)\|_{L^3}^2$, and hence $\|u^{\epsilon,\delta}_0\|_{L^\infty}^2$ and $\|\nabla u^{\epsilon,\delta}_0\|_{L^3}^2$, are of order ${\mathcal O}(1)$. Moreover, (\ref{eq:ass1_c}) and (\ref{eq:ass2_b}) imply for the 
modulated energy (\ref{eq:modulated_energy}) that 
\[ F^{\epsilon,\delta}(0)\le C\delta^{\frac{1}{2}}, \]
because $\delta^2 \|\partial_t \omega^{\epsilon,\delta}(0,\cdot)\|_{L^2}^2 = \delta^2 \|\curl\dive U^{\epsilon,\delta}(0,\cdot)\|_{L^2}^2 \le C\delta^2 \|\dive U^{\epsilon,\delta}(0,\cdot)\|_{H^1}^2 \le C\delta^{\frac{1}{2}}$. Therefore, one can select a $K$-dependent number $M\ge C_1(C+1)$ such that (\ref{assump_a}) and (\ref{impose_a}) hold. On the other hand, the theorem assumptions imply that $\mu={\mathcal O}(\delta)$ for (\ref{assump_b}) and (\ref{assump_c}), which admits sufficiently small $\mu$ to satisfy (\ref{impose_b}) at $t=0$. 

Using the bootstrap assumptions, we perform a priori estimates to recover those assumptions with stricter bounds, which will guarantee the existence of approximate solutions $(u^{\epsilon,\delta},p^{\epsilon,\delta},U^{\epsilon,\delta})$ on $[0,T_0]$, while leading to the desired bounds (\ref{eq:bounds_critical}) at the same time.

For this purpose, we first estimate the time derivative of the modulated energy $F^{\epsilon,\delta}(t)$ by controlling (\ref{eq:modu_energy_evolve}) term by term. We show  at the end of the proof that
\begin{align} 
    |T^{\epsilon,\delta}_\text{I}| &\le \int_{\mathbb T^3} \delta \left|\partial_t\omega^{\epsilon,\delta}\right|^2\, dx + \frac{2}{10} \|\nabla (\omega^{\epsilon,\delta}-\omega)\|_{L^2}^2 + C(K+M)\mu\delta^{\frac{1}{2}}, \label{eq:dFdt_I} \\
    |T^{\epsilon,\delta}_\text{II}| &\le \frac{7}{10}\|\nabla(\omega^{\epsilon,\delta}-\omega)\|_{L^2}^2 + C(K+M)\mu\delta^{\frac{1}{2}} \label{eq:dFdt_II}
\end{align}
hold for sufficiently small $\delta$. Then, (\ref{eq:modu_energy_evolve}) immediately implies
\[
    \frac{d}{dt}F^{\epsilon,\delta}(t) \le C(K+M)\mu\delta^{\frac{1}{2}}
\]
for $\delta\ll 1$. Further by (\ref{eq:impose}), we have
\[
    \sup_{t\in[0,T_0]} F^{\epsilon,\delta}(t) \le \left(\frac{M}{C_1}+C\right)\delta^{\frac{1}{2}}.
\]
By the definition of $F(t)$ in (\ref{eq:modulated_energy}), this would immediately yield the bound (\ref{eq:bounds_critical_curl}) once the bootstrap argument is closed. For now, we conclude 
\begin{equation} \label{eq:curlu_est}
    \left\|\curl(u^{\epsilon,\delta} - u)\right\|_{L^2}^2 + \delta\left\|\curl(u^{\epsilon,\delta} - u)\right\|_{H^1}^2 \le \left(\frac{M}{C_1}+C\right)\delta^{\frac{1}{2}}.
\end{equation}

With $\curl (u^{\epsilon,\delta}-u)$ at hand, we proceed to recover $\dive u^{\epsilon,\delta}$ in (\ref{eq:bounds_critical_div}) and hence $u^{\epsilon,\delta}-u$ in (\ref{eq:bounds_critical_u}). Recall the rescaled residual $W=(q,v,V)$ in (\ref{eq:res_def}). However, here we need to use the initial layer corrections by defining
\begin{subequations} \label{eq:init_layer_correct}
\begin{align}
    V_1(t,\cdot) &:= V(t,\cdot) - e^{-\frac{t}{\delta}}V(0,\cdot), \\
    U_1^{\epsilon,\delta}(t,\cdot) &:= \partial_t U^{\epsilon,\delta}(t,\cdot) - e^{-\frac{t}{\delta}} \partial_t U^{\epsilon,\delta}(0,\cdot). \label{eq:init_correct_U1}
\end{align}
\end{subequations}
Obviously, we have $V_1(0,\cdot)=U_1^{\epsilon,\delta}(0,\cdot)=0$. The variables $W_1:=(v,q,V_1)$ and $(\partial_t u^{\epsilon,\delta},  \partial_t p^{\epsilon,\delta}, U^{\epsilon,\delta}_1)$ are governed by
\begin{equation} \label{eq:init_layer_1}
\left\{
\begin{aligned}
    & \partial_t q + \frac{1}{\sqrt\epsilon} \dive v = -\sqrt\epsilon \partial_t p, \\
    & \partial_t v+ \frac{1}{\sqrt\epsilon}\nabla q + \frac{1}{\sqrt\delta} \dive V_1 = -e^{-\frac{t}{\delta}}\dive ((U^{\epsilon,\delta}-U)(0,\cdot)), \\
    & \partial_t V_1 + \frac{1}{\sqrt\delta} \nabla v = \frac{2uv+v\otimes v}{\sqrt\delta} - \sqrt\delta \partial_t U -\frac{V_1}{\delta},
\end{aligned}
\right.
\end{equation}
with $uv := (u\otimes v+v\otimes u)/2$, and
\begin{equation} \label{eq:init_layer_2}
\left\{
\begin{aligned}
    &\epsilon \partial_t \left(\partial_t p^{\epsilon,\delta}\right) + \dive \left(\partial_t u^{\epsilon,\delta}\right) = 0, \\
    &\partial_t \left(\partial_t u^{\epsilon,\delta}\right) + \nabla \partial_t p^{\epsilon,\delta} + \dive U_1^{\epsilon,\delta} = -e^{-\frac{t}{\delta}}\dive \partial_t U^{\epsilon,\delta}(0,\cdot), \\
    &\delta \partial_t U_1^{\epsilon,\delta} + \nabla \partial_t u^{\epsilon,\delta} = \partial_t \left(u^{\epsilon,\delta}\otimes u^{\epsilon,\delta}\right) - U_1^{\epsilon,\delta},
\end{aligned}
\right.
\end{equation}
respectively. 
Then, an energy is defined as
\begin{equation} \label{eq:energy_crit_e}
    E^{\epsilon,\delta}(t) = \|W_1(t,\cdot)\|_{L^2}^2 + \delta^2 \left\| \left( \partial_t u^{\epsilon,\delta}, \sqrt\epsilon \partial_t p^{\epsilon,\delta}, \sqrt\delta U_1^{\epsilon,\delta} \right) \right\|_{L^2}^2 + \delta^3 \left\| \nabla\partial_t \left( u^{\epsilon,\delta}, \sqrt\epsilon p^{\epsilon,\delta}, \sqrt\delta U^{\epsilon,\delta} \right) \right\|_{L^2}^2.
\end{equation}
The assumptions on the initial conditions (\ref{ass1}) and (\ref{ass2}) imply that
\begin{equation} \label{eq:energy_int_e}
    E^{\epsilon,\delta}(0) = {\mathcal O}(\delta^{\frac{3}{2}}),
\end{equation}
which can be verified by examining (\ref{eq:energy_crit_e}) at $t=0$ term by term. The most important step in the course is to check
\begin{equation}  \label{eq:nabla_dt_U0}
    \|\nabla \partial_t U_0^{\epsilon,\delta}\|_{L^2} = \mathcal O(\delta^{-\frac{5}{4}}),
\end{equation}
since we see from (\ref{eq:relax_c}) that
\[
\begin{aligned}
    \delta \|\nabla \partial_t U_0^{\epsilon,\delta}\|_{L^2} 
    &\le C\left( \| u_0^{\epsilon,\delta}\nabla u_0^{\epsilon,\delta} \|_{L^2} + \|\nabla U_0^{\epsilon,\delta}\|_{L^2} + \|\nabla^2 u_0^{\epsilon,\delta}\|_{L^2} \right) \\
    &\le C\delta^{-\frac{1}{4}},
\end{aligned}
\]
since the initial conditions assure that $\|u_0^{\epsilon,\delta}\|_{L^\infty}$ and $\|\nabla u_0^{\epsilon,\delta}\|_{L^2}$ are both of the order $\mathcal O(1)$.

An energy estimate that will be proven at the end of the proof  yields
\begin{equation} \label{eq:energy_ineq_btstp}
\begin{split}
    \frac{d}{dt}E^{\epsilon,\delta}(t) 
    &\le C(\epsilon+\delta^2)K + C \delta^{-\frac{1}{4}}e^{-\frac{t}{\delta}} \sqrt{E^{\epsilon,\delta}(t)} + C(M+K)E^{\epsilon,\delta}(t) \\
    &\le C\delta^2K + C \delta^{-\frac{1}{4}}e^{-\frac{t}{\delta}} \sqrt{E^{\epsilon,\delta}(t)} + C(M+K)E^{\epsilon,\delta}(t).
\end{split}
\end{equation}
Integrating (\ref{eq:energy_ineq_btstp}), we obtain
\[
   \sup_{0\le s\le t}E^{\epsilon,\delta}(s) \le E^{\epsilon,\delta}(0) + CKt\delta^2 + C\delta^{\frac{3}{4}} \sup_{0\le s\le t} \sqrt{E^{\epsilon,\delta}(s)} + C(M+K) \int_0^t E^{\epsilon,\delta}(s) ds.
\]
Then, by Young's inequality (that $C\delta^{\frac{3}{4}} \sup \sqrt{E^{\epsilon,\delta}(s)}\le \frac{1}{2}  \sup E^{\epsilon,\delta}(s) + 2C^2\delta^{\frac{3}{2}}$), (\ref{eq:energy_int_e}), and Gronwall's inequality, we conclude
\begin{equation} \label{eq:energy_est_btstp}
    E^{\epsilon,\delta}(t) \le C\exp(C(M+K)T_0)\delta^{\frac{3}{2}} =:\mathcal C(C_0, M,K,T_0) \delta^{\frac{3}{2}}
\end{equation}
for all $t\in [0,T_0]$,

As an immediate consequence of (\ref{eq:energy_est_btstp}), the error estimate (\ref{assump_c}) is recovered as
\[
    \| (u^{\epsilon,\delta}-u)(t,\cdot)\|_{L^2}^2 \le \mu \delta^{\frac{1}{2}+r}
\]
with $r\in(0,1)$ and
\[
    \mu = \mathcal C(C_0,M,K,T_0)\delta^{1-r}.
\]
Thus, (\ref{impose_b}) can be satisfied with sufficiently small $\delta\ll 1$. With such $\mu$, (\ref{eq:energy_est_btstp}) further implies that
\[
    \delta^{-\frac{1}{2}} {\|\dive u^{\epsilon,\delta}(t,\cdot)\|}_{L^2}^2 + \delta^{\frac{1}{2}}{ \|\nabla \dive u^{\epsilon,\delta}(t,\cdot)\|}_{L^2}^2 = \mu \delta^r,
\]
which recovers (\ref{assump_b}). The bounds (\ref{eq:bounds_critical_u}) and (\ref{eq:bounds_critical_div}) also follow directly from (\ref{eq:energy_est_btstp}) once the bootstrap argument is closed.

To recover the bootstrap assumption (\ref{assump_a}), recall $v=u^{\epsilon,\delta}-u$ and use (\ref{eq:interp_L3}) to deduce
\[
\begin{aligned}
    \|u^{\epsilon,\delta}(t,\cdot)\|_{L^\infty}^2 
    &\le K + \|v\|_{L^\infty}^2, \\
    \|\nabla u^{\epsilon,\delta}(t,\cdot)\|_{L^3}^2 &\le K + \|\nabla v\|_{L^3}^2 
    \le K + C \|\nabla v\|_{L^2} \|\nabla v\|_{H^1}.
\end{aligned}
\]
Consider the decomposition  $v=\bar v + \tilde v$ with $\bar v=\frac{1}{\text{vol}(\mathbb T^3)}\int_{\mathbb T^3} v(t,x) \,dx  = \frac{1}{\text{vol}(\mathbb T^3)}\int_{\mathbb T^3} v(0,x) \,dx$ because $\frac{d}{dt}\int_{\mathbb T^3} v \,dx=0$ by (\ref{eq:ns}) and (\ref{eq:relax}).
Then the Poincar{\'e} inequality implies that $\|\tilde v\|_{L^2}\le C\|\nabla v\|_{L^2} \le  C\|\nabla^2 v\|_{L^2}$. It follows that
\[
\begin{aligned}
    \|v\|_{L^\infty}^2 \le |\bar v|^2 + \|\tilde v\|_{L^\infty}^2
    &\le  C \left( \|v(0,\cdot)\|_{L^2}^2 + \|\tilde v\|_{H^1} \|\tilde v\|_{H^2} \right) \\
    &\le C \left( \|v(0,\cdot)\|_{L^2}^2 + \|\nabla v\|_{L^2} \|\nabla^2 v\|_{L^2} \right)
\end{aligned}
\]
and hence
\[
\begin{aligned}
    \|u^{\epsilon,\delta}(t,\cdot)\|_{L^\infty}^2 + \|\nabla u^{\epsilon,\delta}(t,\cdot)\|_{L^3}^2 
    &\le 2K + C \left( \|v(0,\cdot)\|_{L^2}^2 + \|\nabla v\|_{L^2} \|\nabla^2 v\|_{L^2} \right) \\
    &\le 2K+ C\delta^{\frac{3}{2}} + C\|\nabla v\|_{L^2} \|\nabla^2 v\|_{L^2}.
\end{aligned}
\]
Thanks to (\ref{eq:hodge}), it holds that $\|\nabla^\ell v\|_{L^2}\le C\left(\|\nabla^{\ell-1}\curl v\|_{L^2} + \|\nabla^{\ell-1}\dive v\|_{L^2}\right)$ for $\ell=1,2$. One sees from (\ref{assump_b}) and (\ref{eq:curlu_est}) that $\|\nabla v\|_{L^2} = {\mathcal O}(\delta^{1/4})$ and $\|\nabla^2 v\|_{L^2}={\mathcal O}(\delta^{-1/4})$, with the constants essentially determined by $\|\curl v\|_{L^2}$ and $\|\nabla\curl v\|_{L^2}$, respectively (both being $\sqrt{M/C_1+C}$). Note that the  ``$\dive u$" counterpart $\mu$ is of order ${\mathcal O}(\delta)$. As a consequence, the product $\|\nabla v\|_{L^2} \|\nabla^2 v\|_{L^2}$ is of ${\mathcal O}(1)$:
\[
    \|u^{\epsilon,\delta}(t,\cdot)\|_{L^\infty}^2 + \|\nabla u^{\epsilon,\delta}(t,\cdot)\|_{L^3}^2 \le 2K + C \left( \frac{M}{C_1}+C \right).
\]
Note that in this expression, $C$ is a generic, sufficiently large constant, dependent only on $C_0$ and $\mathbb T^3$, and can be fixed \textit{a priori}. Thus, if we take
\[
    C_1>2C, \quad
    M>(C^2+2K) \left(\frac{1}{2}-\frac{1}{C_1/C}\right)^{-1}
\]
in the bootstrap assumption (\ref{assump_a}) and (\ref{impose_a}), then we actually get
\[
    \|u^{\epsilon,\delta}(t,\cdot)\|_{L^\infty}^2 + \|\nabla u^{\epsilon,\delta}(t,\cdot)\|_{L^3}^2 \le \frac{M}{2}.
\]
In this way, the bootstrap argument is fully closed, which shows that the solutions $(u^{\epsilon,\delta},p^{\epsilon,\delta},U^{\epsilon,\delta})$ do exist on $[0,T_0]$, and the desired bounds (\ref{eq:bounds_critical}) follow directly from (\ref{eq:curlu_est}) and (\ref{eq:energy_est_btstp}).

To complete the proof, it remains to verify (\ref{eq:dFdt_I}), (\ref{eq:dFdt_II}), and (\ref{eq:energy_ineq_btstp}). 

\textbf{Verification of (\ref{eq:dFdt_I}).} It holds that
\[
      |T^{\epsilon,\delta}_\text{I}| = \left|\int_{\mathbb T^3} \delta \partial_t\omega^{\epsilon,\delta} \cdot \left( A-\Delta \omega - 2A^{\epsilon,\delta} \right)\,dx\right|
    \le \int_{\mathbb T^3} \delta \left|\partial_t\omega^{\epsilon,\delta}\right|^2\, dx + \delta {\| A-\Delta \omega\|}_{L^2}^2 + 2\delta\left\| A^{\epsilon,\delta} \right\|_{L^2}^2,
\]
and
\[
\begin{aligned}
    |A^{\epsilon,\delta}|^2  &\le C \left(|\nabla\dive u^{\epsilon,\delta}|^2 + |\nabla(\omega^{\epsilon,\delta}-\omega)|^2 + |\nabla\omega|^2 \right) |u^{\epsilon,\delta}|^2 \\[1.2ex]
    &\qquad \qquad + C|\omega^{\epsilon,\delta}-\omega|^2 |\nabla u^{\epsilon,\delta}|^2 + C|\omega|^2 |\nabla u^{\epsilon,\delta}|^2,
\end{aligned}
\]
due to (\ref{eq:A_def}). Then we have a control of $\|A^{\epsilon,\delta}\|_{L^2}$ as
\[
\begin{aligned}
    \delta \|A^{\epsilon,\delta}\|_{L^2}^2 &\le CM\mu\delta^{\frac{1}{2}} + CM\delta\|\nabla (\omega^{\epsilon,\delta}-\omega)\|_{L^2}^2 + CM\delta \|\omega\|_{H^1}^2 \\
    &\le C\delta^{\frac{1}{2}}\left( M\mu + MK\delta^{\frac{1}{2}} \right) + \frac{1}{10} \|\nabla (\omega^{\epsilon,\delta}-\omega)\|_{L^2}^2
\end{aligned}
\]
for $\delta\ll 1$. Here, the second inequality results from the facts that $\|\omega\|_{H^1}^2< K$ and $CM\delta\le 1/10$ for sufficiently small $\delta$. For the first inequality, the control of the term 
\[
    \int_{\mathbb T^3} \left(|\nabla\dive u^{\epsilon,\delta}|^2 + |\nabla(\omega^{\epsilon,\delta}-\omega)|^2 + |\nabla\omega|^2 \right) |u^{\epsilon,\delta}|^2\,dx
\]
is obvious, and the control of the rest part $\int_{\mathbb T^3} \left( |\omega^{\epsilon,\delta}-\omega|^2 |\nabla u^{\epsilon,\delta}|^2 + |\omega|^2 |\nabla u^{\epsilon,\delta}|^2 \right) \,dx$ uses (\ref{eq:sob_emb}) and H{\"o}lder's inequality, yielding
\[
\begin{aligned}
    \delta\int_{\mathbb T^3} |\omega^{\epsilon,\delta}-\omega|^2 |\nabla u^{\epsilon,\delta}|^2 \,dx
    &\le C\delta \left\| |\omega^{\epsilon,\delta}-\omega|^2 \right\|_{L^3} \left\| |\nabla u^{\epsilon,\delta}|^2 \right\|_{L^{\frac{3}{2}}} \\
    &=C\delta \|\omega^{\epsilon,\delta}-\omega\|_{L^6}^2 \|\nabla u^{\epsilon,\delta}\|_{L^3}^2
    \le CM\delta \|\nabla(\omega^{\epsilon,\delta}-\omega)\|_{L^2}^2
\end{aligned}
\]
and similarly,
\[
\begin{aligned}
    \delta \int_{\mathbb T^3} |\omega|^2 |\nabla u^{\epsilon,\delta}|^2 \,dx
    &\le C\delta \|\omega\|_{L^6}^2 \|\nabla u^{\epsilon,\delta}\|_{L^3}^2 \le CM\delta \|\omega\|_{H^1}^2.
\end{aligned}
\]
As a consequence, (\ref{eq:dFdt_I}) follows.

\textbf{Verification of (\ref{eq:dFdt_II}).}
To handle $T^{\epsilon,\delta}_\text{II}(t)=-\int_{\mathbb T^3} (\omega^{\epsilon,\delta}-\omega)\cdot(A^{\epsilon,\delta}-A)\,dx$, we resort to the decomposition
\begin{equation} \label{eq:decomp}
\begin{split}
    A^{\epsilon,\delta} - A =& \curl \left( u^{\epsilon,\delta}\dive u^{\epsilon,\delta} \right) + \omega^{\epsilon,\delta} \dive u^{\epsilon,\delta} + u^{\epsilon,\delta} \cdot \nabla (\omega^{\epsilon,\delta} - \omega) \\
    &+ (u^{\epsilon,\delta}-u)\cdot\nabla \omega - \left( \omega^{\epsilon,\delta}\cdot\nabla u^{\epsilon,\delta} - \omega\cdot\nabla u \right),
\end{split}
\end{equation}
which is a direct consequence of (\ref{eq:A0_def}) and (\ref{eq:A_def}). It follows that $|T^{\epsilon,\delta}_\text{II}|\le \sum_{i=1}^5\Pi_i$ with
\[
\begin{aligned}
   \Pi_1 &:=\left | \int_{\mathbb T^3} (\omega^{\epsilon,\delta}-\omega)\cdot \curl\left( u^{\epsilon,\delta}\dive u^{\epsilon,\delta} \right) \, dx\right|
    = \left | \int_{\mathbb T^3} \curl (\omega^{\epsilon,\delta}-\omega)\cdot \left( u^{\epsilon,\delta}\dive u^{\epsilon,\delta} \right) \, dx\right| \\
    &\le \|\curl (\omega^{\epsilon,\delta}-\omega)\|_{L^2} {\| u^{\epsilon,\delta}\dive u^{\epsilon,\delta}\|}_{L^2}
    \le \|\nabla(\omega^{\epsilon,\delta}-\omega)\|_{L^2} \left( \|u^{\epsilon,\delta}\|_{L^\infty} {\| \dive u^{\epsilon,\delta}\|}_{L^2} \right) \\
    &\le \frac{1}{10} \|\nabla(\omega^{\epsilon,\delta}-\omega)\|_{L^2}^2 + \frac{5}{2} \|u^{\epsilon,\delta}\|_{L^\infty}^2{} {\| \dive u^{\epsilon,\delta}\|}_{L^2}^2 \\
    &\le \frac{1}{10} \|\nabla(\omega^{\epsilon,\delta}-\omega)\|_{L^2}^2 + C M\mu\delta^{\frac{1}{2}},
\end{aligned}
\]
where the second equality results from integration by parts. H{\"o}lder's inequality and Young's inequality are used in the second and third line, respectively. Likewise,
\[
\begin{aligned}
   \Pi_2 &:=\left| \int_{\mathbb T^3} (\omega^{\epsilon,\delta}-\omega)\cdot \omega^{\epsilon,\delta}\dive u^{\epsilon,\delta} \,dx \right|
    \le \int_{\mathbb T^3} \left| \omega^{\epsilon,\delta}-\omega \right|^2 |\dive u^{\epsilon,\delta}| \,dx + \int_{\mathbb T^3} \left| \omega^{\epsilon,\delta}-\omega \right| \cdot \left|\omega \dive u^{\epsilon,\delta}\right| \,dx \\
    &\le \|\omega^{\epsilon,\delta}-\omega\|_{L^6} \left(\|\omega^{\epsilon,\delta}-\omega\|_{L^3} + \|\omega\|_{L^3} \right) {\|\dive u^{\epsilon,\delta}\|}_{L^2} \\
    &\le C \|\nabla(\omega^{\epsilon,\delta}-\omega)\|_{L^2} \left(\|\omega^{\epsilon,\delta}-\omega\|_{L^3} + \|\omega\|_{L^3} \right) {\|\dive u^{\epsilon,\delta}\|}_{L^2} \\
    &\le \frac{1}{10}\|\nabla(\omega^{\epsilon,\delta}-\omega)\|_{L^2}^2 + C(M+K)\mu\delta^{\frac{1}{2}}, 
\end{aligned}
\]
where (\ref{eq:sob_emb}) was used in the third line, and the $L^3$-norms are bounded because of 
\[
    \|\omega^{\epsilon,\delta}-\omega\|_{L^3}^2 + \|\omega\|_{L^3}^2 \le \|\omega^{\epsilon,\delta}\|_{L^3}^2 + 2\|\omega\|_{L^3}^2 \le C\left( \|\nabla u^{\epsilon,\delta}\|_{L^3}^2 + \|\omega\|_{H^1}^2 \right) \le C(M+K),
\]
due to (\ref{assump_a}) and the Sobolev embedding $H^1(\mathbb T^3) \hookrightarrow L^3(\mathbb T^3)$. 

The third term is treated as
\[
\begin{aligned}
   \Pi_3 &:= \left| \int_{\mathbb T^3} (\omega^{\epsilon,\delta}-\omega)\cdot \left[ u^{\epsilon,\delta} \cdot \nabla (\omega^{\epsilon,\delta} - \omega) \right] \,dx \right|
    = \left| \int_{\mathbb T^3} u^{\epsilon,\delta}\cdot \nabla \frac{|\omega^{\epsilon,\delta}-\omega|^2}{2} \,dx \right| \\
    &= \left| \int_{\mathbb T^3} \frac{|\omega^{\epsilon,\delta}-\omega|^2}{2} \dive u^{\epsilon,\delta} \,dx\right|,
\end{aligned}
\]
so it can be controlled in the same way as $\text{II}_2$. For the fourth term, we have
\[
\begin{aligned}
   \Pi_4 &:= \left| \int_{\mathbb T^3} (\omega^{\epsilon,\delta}-\omega)\cdot \left[ (u^{\epsilon,\delta}-u)\cdot\nabla \omega \right] \,dx \right|
    \le {\|\omega^{\epsilon,\delta} -\omega\|}_{L^6} \|u^{\epsilon,\delta}-u\|_{L^2} \|\nabla \omega\|_{L^3} \\
    &\le C \|\nabla(\omega^{\epsilon,\delta}-\omega)\|_{L^2} \|u^{\epsilon,\delta}-u\|_{L^2} \|\nabla \omega\|_{L^3} \\
    &\le \frac{1}{10} \|\nabla(\omega^{\epsilon,\delta}-\omega)\|_{L^2}^2 + C \|u^{\epsilon,\delta}-u\|_{L^2}^2 \|\nabla \omega\|_{L^3}^2  \\
    &\le \frac{1}{10} \|\nabla(\omega^{\epsilon,\delta}-\omega)\|_{L^2}^2 + CK\mu\delta^{\frac{1}{2}},
\end{aligned}
\]
by using (\ref{eq:sob_emb}), (\ref{assump_c}) and $\|\nabla \omega\|_{L^3}\le C\|\nabla \omega\|_{H^1}\le C\|u\|_{H^3}\le CK$ in the last line.

The fifth term is split as follows
\[
\begin{aligned}
   \Pi_5&:= \left| \int_{\mathbb T^3} (\omega^{\epsilon,\delta}-\omega)\cdot \left( \omega^{\epsilon,\delta}\cdot\nabla u^{\epsilon,\delta} - \omega\cdot\nabla u \right) \,dx \right| \\
    &\le \left| \int_{\mathbb T^3} (\omega^{\epsilon,\delta}-\omega)\cdot \left[ (\omega^{\epsilon,\delta}-\omega)\cdot \nabla u \right] \,dx \right| + \left| \int_{\mathbb T^3} (\omega^{\epsilon,\delta}-\omega)\cdot \left[ \omega^{\epsilon,\delta} \cdot \nabla (u^{\epsilon,\delta}-u) \right] \,dx \right| \\
    &=:\Pi_{5}^{(1)} +\Pi_{5}^{(2)}.
\end{aligned}
\]
Then, recalling $\omega=\curl u$, $\omega^{\epsilon,\delta}=\curl u^{\epsilon,\delta}$ and integrating by parts, we have 
\[
\begin{aligned}
   \Pi_{5}^{(1)} &= \left| \int_{\mathbb T^3} (u^{\epsilon,\delta}-u)\cdot \curl \left[ (\omega^{\epsilon,\delta}-\omega)\cdot \nabla u \right] \,dx \right| \\
    &\le {\| u^{\epsilon,\delta} - u\|}_{L^2} {\| \curl \left[ (\omega^{\epsilon,\delta}-\omega)\cdot \nabla u \right]\|}_{L^2} \\
    &\le \frac{1}{10} {\| \nabla (\omega^{\epsilon,\delta}-\omega)\|}_{L^2}^2 + \frac{5}{2}C{ \|\nabla u\|}_{L^\infty}^2 {\| u^{\epsilon,\delta} - u\|}_{L^2}^2  \\
    &\le \frac{1}{10} {\| \nabla (\omega^{\epsilon,\delta}-\omega)\|}_{L^2}^2 + CK\mu\delta^{\frac{1}{2}}
\end{aligned}
\]
and (by using $\dive \omega^{\epsilon,\delta}=0$)
\[ 
\begin{aligned}
   \Pi_{5}^{(2)} &= \left| \int_{\mathbb T^3} (u^{\epsilon,\delta}-u) \cdot \left[ \omega^{\epsilon,\delta} \cdot \nabla(\omega^{\epsilon,\delta}-\omega) \right] \,dx \right| \\
    &\le \|u^{\epsilon,\delta}-u\|_{L^3} \|\omega^{\epsilon,\delta}-\omega\|_{L^6} \|\nabla (\omega^{\epsilon,\delta}-\omega)\|_{L^2} + \|u^{\epsilon,\delta}-u\|_{L^2} \|\omega\|_{L^\infty} \|\nabla(\omega^{\epsilon,\delta}-\omega)\|_{L^2}.
\end{aligned}
\]
By (\ref{eq:interp_L3}), it holds that 
\[
\begin{aligned}
    \|u^{\epsilon,\delta}-u\|_{L^3}^2 &\le C\|u^{\epsilon,\delta}-u\|_{L^2}^2 +  C\|u^{\epsilon,\delta}-u\|_{L^2} \left( \|\nabla u^{\epsilon,\delta}\|_{L^2} + \|\nabla u\|_{L^2} \right) \\
    &\le C\mu\delta^{\frac{1}{2}} + C\mu^{\frac{1}{2}}\delta^{\frac{1}{4}}(M^{\frac{1}{2}}+K^{\frac{1}{2}}) \\
    &\le C\mu^{\frac{1}{2}} (M^{\frac{1}{2}}+K^{\frac{1}{2}}) \delta^{\frac{1}{4}} 
\end{aligned}
\]
for $\delta\ll 1$ because of $\|\nabla u^{\epsilon,\delta}\|_{L^2(\mathbb T^3)} \le C(\text{vol}(\mathbb T^3)) \|\nabla u^{\epsilon,\delta}\|_{L^3(\mathbb T^3)}$ in the second line. Hence, we have
\[
\begin{aligned}
   \Pi_{5}^{(2)} &\le C \mu^{\frac{1}{4}} (M+K)^{\frac{1}{4}}\delta^{\frac{1}{8}} \|\nabla (\omega^{\epsilon,\delta}-\omega)\|_{L^2}^2 + \frac{1}{10} \|\nabla(\omega^{\epsilon,\delta}-\omega)\|_{L^2}^2 + CK\mu\delta^{\frac{1}{2}} \\
    &\le \frac{2}{10} \|\nabla(\omega^{\epsilon,\delta}-\omega)\|_{L^2}^2 + CK\mu\delta^{\frac{1}{2}}
\end{aligned}
\]
for $\delta \ll 1$.

Collecting the above estimates for $\Pi_i$ immediately gives (\ref{eq:dFdt_II}).

\textbf{Verification of (\ref{eq:energy_ineq_btstp}).} First, we work with the system (\ref{eq:init_layer_1}) for the initial-layer-corrected variable $W_1$ and use the bootstrap assumption (\ref{assump_a}) to get
\[
\begin{aligned}
    \frac{1}{2}\frac{d}{dt}\|W_1\|_{L^2}^2 
    &= \int_{\mathbb T^3} \left[-(\sqrt\epsilon \partial_tp)q - v\cdot e^{-\frac{t}{\delta}}\dive ((U^{\epsilon,\delta}-U)(0,\cdot)) \right] \,dx \\
    &\quad \ +\int_{\mathbb T^3} \frac{V_1}{\sqrt\delta}:\left( 2uv + v\otimes v - \delta \partial_t U \right) \, dx - \int_{\mathbb T^3} \frac{|V_1|^2}{\delta}\,dx \\
    &\le C \left( \sqrt{\epsilon K} \|q\|_{L^2} + (\|u\|_{L^\infty}^2+\|v\|_{L^\infty}^2) \|v\|_{L^2}^2 + K\delta^2 \right) + e^{-\frac{t}{\delta}}\|v\|_{L^2} \|\dive((U^{\epsilon,\delta}-U)(0,\cdot))\|_{L^2} \\
    &\le C(\epsilon+\delta^2)K + C(M+K)E^{\epsilon,\delta}(t) + C \delta^{-\frac{1}{4}}e^{-\frac{t}{\delta}} \sqrt{E^{\epsilon,\delta}(t)}.
\end{aligned}
\]

Then, the system (\ref{eq:init_layer_2}) for the variable $(\partial_t u^{\epsilon,\delta},  \partial_t p^{\epsilon,\delta}, U^{\epsilon,\delta}_1)$ can be processed as
\[
\begin{aligned}
    \delta^2 \frac{d}{dt} \frac{1}{2} &\| ( \partial_tu^{\epsilon,\delta}, \sqrt\epsilon \partial_tp^{\epsilon,\delta}, \sqrt\delta U_1^{\epsilon,\delta})\|_{L^2}^2 \\
    &= - \delta^2 e^{-\frac{t}{\delta}} \int_{\mathbb T^3} \partial_tu^{\epsilon,\delta} \cdot \dive \partial_t U^{\epsilon,\delta}_0 \,dx  + \delta^2\int_{\mathbb T^3} \partial_t U^{\epsilon,\delta}: \left(  \partial_t (u^{\epsilon,\delta}\otimes u^{\epsilon,\delta}) - \partial_t U^{\epsilon,\delta} \right)  \, dx \\
    &\le e^{-\frac{t}{\delta}} \left(\delta\|\partial_t u^{\epsilon,\delta}\|_{L^2}\right) \left(\delta \| \dive \partial_t U^{\epsilon,\delta}_0 \|_{L^2}\right) +  C\delta^2 \| u^{\epsilon,\delta}\otimes \partial_tu^{\epsilon,\delta}\|_{L^2}^2 \\
    &\le C \delta^{-\frac{1}{4}}e^{-\frac{t}{\delta}} \sqrt{E^{\epsilon,\delta}(t)} + CME^{\epsilon,\delta}(t),
\end{aligned}
\]
where we used (\ref{eq:nabla_dt_U0}) and $\delta^2 {\| u^{\epsilon,\delta}\otimes \partial_t u^{\epsilon,\delta}\|}_{L^2}^2 \le {\| u^{\epsilon,\delta}\|}_{L^\infty}^2 \cdot \delta^2{\| \partial_t u^{\epsilon,\delta}\|}_{L^2}^2 \le ME^{\epsilon,\delta}(t)$ in the last line.

Similarly, letting $\nabla \partial_t$ act on (\ref{eq:relax}), we have
\[
\begin{aligned}
    \delta^3\frac{d}{dt}\frac{1}{2}{\| \nabla\partial_t ( u^{\epsilon,\delta}, \sqrt\epsilon p^{\epsilon,\delta}, \sqrt\delta U^{\epsilon,\delta})\|}_{L^2}^2
    &\le \delta^3{\| \nabla\partial_t (u^{\epsilon,\delta}\otimes u^{\epsilon,\delta})\|}_{L^2}^2 \le C \delta^3 {\| \nabla (u^{\epsilon,\delta}\otimes \partial_tu^{\epsilon,\delta})\|}_{L^2}^2 \\
    &\le C\delta^3 \left({\|\nabla u^{\epsilon,\delta} \otimes \partial_t u^{\epsilon,\delta}\|}_{L^2}^2 +{ \| u^{\epsilon,\delta} \otimes \nabla \partial_t u^{\epsilon,\delta}\|}_{L^2}^2 \right) \\
    &\le C\delta^3 \left( {\|\nabla u^{\epsilon,\delta}\|}_{L^3}^2 {\| \partial_t u^{\epsilon,\delta}\|}_{L^6}^2 +  {\| u^{\epsilon,\delta}\|}_{L^\infty}^2{\| \nabla\partial_t u^{\epsilon,\delta}\|}_{L^2}^2 \right) \\
    &\le CME^{\epsilon,\delta}(t),
\end{aligned}
\]
by using the bootstrap assumption (\ref{assump_a}) and the Sobolev embedding $H^1(\mathbb T^3)\hookrightarrow L^6(\mathbb T^3)$.

With the above three estimates, (\ref{eq:energy_ineq_btstp}) follows immediately. This also completes the whole proof.

\begin{remark}
We point out two important technical features. First, in our current approach, the scaling $\epsilon=\mathcal O(\delta^2)$ plays an essential role in deriving (\ref{eq:energy_int_e}) and (\ref{eq:energy_ineq_btstp}). Second, the implementation of the initial layer correction (\ref{eq:init_layer_correct}) for $U^{\epsilon,\delta}$ allows us to remove the stringent constraint that $U^{\epsilon,\delta}_0$ must be $o(1)$ close to $u_0 \otimes u_0 - \nabla u_0$ and hence stablish convergence for a broader class of initial data.
\end{remark}

\subsection{Proof of Theorem~\ref{thm:critical_3d_inhomo} on $\mathbb T^3$} \label{subsec:thm:critical_3d_inhomo}

The proof uses a similar bootstrap argument to that of Theorem~\ref{thm:critical_3d}. We mainly remark the differences. Denote 
\[
    K=1+\|u\|_{C^0H^3}^2,
\]
and $C$ still denotes a generic positive constant, dependent only on $C_0$ and the domain $\mathbb T^3$, that can be fixed a priori.

Besides (\ref{eq:relation_3d_inhomo}), we introduce additional bootstrap assumptions as
\begin{subequations} \label{eq:assump_ext}
\begin{align}
    &\|u^{\epsilon,\delta}(t,\cdot)\|_{L^\infty}^2 + {\| \nabla u^{\epsilon,\delta}(t,\cdot)\|}_{L^3}^2 \le M=M(C_0, T_0,K), \label{assump_ext_a} \\
    &\delta^{-\frac{1}{2}}\|\dive u^{\epsilon,\delta}(t,\cdot)\|_{L^2}^2 + \delta^{\frac{1}{2}} {\| \nabla\dive u^{\epsilon,\delta}(t,\cdot)\|}_{L^2}^2 \le \frac{1}{K+M}. \label{assump_ext_b}
\end{align}
\end{subequations}
The assumption (i.e., the existence of a sufficiently-large $M$) holds at $t=0$. With (\ref{eq:assump_ext}) and (\ref{eq:relation_3d_inhomo}), we have the following estimate
\begin{equation} \label{eq:v_L6}
\begin{split}
    \| u^{\epsilon,\delta}-u \|_{L^6}^2 &\le C \left| \int_{\mathbb T^3} (u^{\epsilon,\delta}-u) \,dx \right|^2 + C \left( \| \omega^{\epsilon,\delta}-\omega \|_{L^2}^2 + \| \dive u^{\epsilon,\delta} \|_{L^2}^2 \right) \\
    &\le C\delta^{\frac{1}{2}} + CF^{\epsilon,\delta}(t),
\end{split}
\end{equation}
which will be used in the sequel. Indeed, for the first line, we again decompose $v:=u^{\epsilon,\delta}-u = \bar v+\tilde v$ with $\int_{\mathbb T^3} \tilde v \,dx=0$, and apply the Sobolev embedding $H^1(\mathbb T^3)\hookrightarrow L^6(\mathbb T^3)$ and Poincar{\'e} inequality. The second line is a consequence of the fact that $\frac{d}{dt}\int_{\mathbb T^3} v \,dx=0$, the theorem assumption that $\|v(0,\cdot)\|_{L^2}^2\le \delta^{1/2}$ and (\ref{assump_ext_b}). 

For the a priori estimate, we first handle $\curl u^{\epsilon,\delta}$ using the modulated energy $F^{\epsilon,\delta}$ in (\ref{eq:modulated_energy}). It can be shown that (\ref{eq:dFdt_I}) still holds, while (\ref{eq:dFdt_II}) is subject to changes:
\begin{equation} \label{eq:dFdt_II_new}
    |T^{\epsilon,\delta}_\text{II}| \le \frac{5}{10} \|\nabla(\omega^{\epsilon,\delta}-\omega)\|_{L^2}^2 + CK\delta^{\frac{1}{2}} + CKF^{\epsilon,\delta}(t) + C\left(F^{\epsilon,\delta}(t)\right)^3.
\end{equation}
Indeed, to show (\ref{eq:dFdt_II_new}), we still resort to $|T^{\epsilon,\delta}_\text{II}|\le \sum_{i=1}^5\Pi_i$ from the decomposition (\ref{eq:decomp}), with $\Pi_1$, $\Pi_2$ and $\Pi_3$ treated in the same way as in the proof of Theorem~\ref{thm:critical_3d}. The other terms are now estimated as
\[
\begin{aligned}
\Pi_4 &:= \left| \int_{\mathbb T^3} (\omega^{\epsilon,\delta}-\omega)\cdot \left[ (u^{\epsilon,\delta}-u)\cdot\nabla \omega \right] \,dx \right|
    \le {\|\omega^{\epsilon,\delta} -\omega\|}_{L^2} {\|u^{\epsilon,\delta}-u\|}_{L^6} \|\nabla \omega\|_{L^3} \\
    &\le {\|\omega^{\epsilon,\delta} -\omega\|}_{L^2}^2 + {\|u^{\epsilon,\delta}-u\|}_{L^6}^2 \|\nabla \omega\|_{L^3}^2
    \le CKF^{\epsilon,\delta}(t) + CK\delta^{\frac{1}{2}}, \\
    \Pi_5^{(1)} &:= \left| \int_{\mathbb T^3} (\omega^{\epsilon,\delta}-\omega)\cdot \left[ (\omega^{\epsilon,\delta}-\omega)\cdot \nabla u \right] \,dx \right| \le K^{\frac{1}{2}} F^{\epsilon,\delta}(t),
\end{aligned}
\]
and
\[ 
\begin{aligned}
   \Pi_{5}^{(2)} &= \left| \int_{\mathbb T^3} (u^{\epsilon,\delta}-u) \cdot \left[ \omega^{\epsilon,\delta} \cdot \nabla(\omega^{\epsilon,\delta}-\omega) \right] \,dx \right| \\
    &\le \|u^{\epsilon,\delta}-u\|_{L^6} \|\omega^{\epsilon,\delta}-\omega\|_{L^3} \|\nabla (\omega^{\epsilon,\delta}-\omega)\|_{L^2} + \|u^{\epsilon,\delta}-u\|_{L^6} \|\omega\|_{L^3} \|\nabla(\omega^{\epsilon,\delta}-\omega)\|_{L^2}.
\end{aligned}
\]
By $\|\omega^{\epsilon,\delta}-\omega\|_{L^3}^2 \le C \|\omega^{\epsilon,\delta}-\omega\|_{L^2} \|\nabla(\omega^{\epsilon,\delta}-\omega)\|_{L^2}$, it holds that
\[
\begin{aligned}
    \|u^{\epsilon,\delta}-u\|_{L^6} \|\omega^{\epsilon,\delta}-\omega\|_{L^3} &\|\nabla (\omega^{\epsilon,\delta}-\omega)\|_{L^2} \le C\|u^{\epsilon,\delta}-u\|_{L^6} \|\omega^{\epsilon,\delta}-\omega\|_{L^2}^{\frac{1}{2}} \|\nabla(\omega^{\epsilon,\delta}-\omega)\|_{L^2}^{\frac{3}{2}} \\
    &\le \frac{1}{10} \left( \|\nabla(\omega^{\epsilon,\delta}-\omega)\|_{L^2}^{\frac{3}{2}} \right)^{\frac{4}{3}} + C \left( \|u^{\epsilon,\delta}-u\|_{L^6} \|\omega^{\epsilon,\delta}-\omega\|_{L^2}^{\frac{1}{2}}\right)^4 \\
    &\le \frac{1}{10} \|\nabla(\omega^{\epsilon,\delta}-\omega)\|_{L^2}^2 + C\left[(F^{\epsilon,\delta}(t))^2+\delta \right] F^{\epsilon,\delta}(t),
\end{aligned}
\]
and hence
\[
\begin{aligned}
    \Pi_{5}^{(2)} &\le \left[ \frac{1}{10} \|\nabla(\omega^{\epsilon,\delta}-\omega)\|_{L^2}^2 + C(F(t)^2+\delta)F(t) \right] + \frac{1}{10} \|\nabla(\omega^{\epsilon,\delta}-\omega)\|_{L^2}^2 + CK(F(t)+\delta^{\frac{1}{2}}) \\
    &\le \frac{2}{10} \|\nabla(\omega^{\epsilon,\delta}-\omega)\|_{L^2}^2 + C\left(F^{\epsilon,\delta}(t)\right)^3 + CKF^{\epsilon,\delta}(t) + CK\delta^{\frac{1}{2}}.
\end{aligned}
\]
These estimates give (\ref{eq:dFdt_II_new}) immediately. Then we have
\[
    \frac{d}{dt}F^{\epsilon,\delta}(t) \le CK\delta^{\frac{1}{2}} + CKF^{\epsilon,\delta}(t) + C\left(F^{\epsilon,\delta}(t)\right)^3.
\]
Since $F^{\epsilon,\delta}(0)\le C\delta^{1/2}\ll 1$, we conclude from a standard argument that
\begin{equation} \label{eq:}
    F^{\epsilon,\delta}(t)\le \mathcal C(C_0, T_0,K)\delta^{\frac{1}{2}}, \quad t\in[0,T_0]
\end{equation}
for sufficiently small $\delta \ll 1$. Here $\mathcal C(C_0,T_0,K)$ is a constant dependent only on $C_0$, $T_0$ and $K$. By the definition of $F^{\epsilon,\delta}(t)$, this implies (\ref{eq:bounds_critcoro_u}) and (\ref{eq:bounds_critcoro_curl}) once the whole bootstrap argument is closed. Indeed, (\ref{eq:bounds_critcoro_u}) is a consequence of (\ref{eq:bounds_critcoro_curl}) and the bootstrap assumption (\ref{assump_ext_b}).

To prove (\ref{eq:bounds_critcoro_div}) and recover the bootstrap assumptions, define
\begin{equation} \label{eq:energy_crit_inhomo}
    E^{\epsilon,\delta}(t) = \left\| \left(\partial_t u^{\epsilon,\delta}, \sqrt\epsilon \partial_t p^{\epsilon,\delta}, \sqrt\delta  U_1^{\epsilon,\delta}\right) \right\|_{L^2}^2 + \delta \left\| \nabla\partial_t \left(u^{\epsilon,\delta}, \sqrt\epsilon p^{\epsilon,\delta}, \sqrt\delta U^{\epsilon,\delta}\right) \right\|_{L^2}^2
\end{equation}
with $U_1^{\epsilon,\delta}$ given in (\ref{eq:init_correct_U1}) for an initial layer correction. Analogous to the derivation of  (\ref{eq:energy_int_e}) and (\ref{eq:energy_ineq_btstp}), here we have 
\begin{equation} \label{eq:energy_ineq_inhomo}
    \frac{d}{dt}E^{\epsilon,\delta}(t) \le CME^{\epsilon,\delta}(t) + C\delta^{-\frac{5}{4}}e^{-\frac{t}{\delta}} \sqrt{E^{\epsilon,\delta}(t)}
\end{equation}
with
\begin{equation} \label{eq:energy_int_inhomo}
    E^{\epsilon,\delta}(0)=\mathcal O(\delta^{-\frac{1}{2}}).
\end{equation}
We further integrate (\ref{eq:energy_ineq_inhomo}) to obtain
\[
   \sup_{0\le s\le t}E^{\epsilon,\delta}(s) \le E^{\epsilon,\delta}(0) + CM \int_0^t E^{\epsilon,\delta}(s) ds + C\delta^{-\frac{1}{4}} \sup_{0\le s\le t} \sqrt{E^{\epsilon,\delta}(s)}.
\]
Then, by Young's inequality (that $C\delta^{-\frac{1}{4}} \sup \sqrt{E^{\epsilon,\delta}(s)}\le \frac{1}{2} \sup E^{\epsilon,\delta}(s) + 2C^2\delta^{-\frac{1}{2}}$), (\ref{eq:energy_int_inhomo}), and Gronwall's inequality, we have
\begin{equation} \label{eq:energy_bound_inhomo}
    E^{\epsilon,\delta}(t) \le C\exp(CMT_0)\delta^{-\frac{1}{2}}
\end{equation}
for all $t\in [0,T_0]$, which immediately implies (\ref{eq:bounds_critcoro_div}) (due to $\epsilon \partial_t p^{\epsilon,\delta}=-\dive u^{\epsilon,\delta}$). The existence of $M=M(C_0, T_0,K)$ in (\ref{assump_ext_a}) can be guaranteed with a similar argument as in the proof of Theorem~\ref{thm:critical_3d}. Given $C$ and $M$, by choosing sufficiently small
\[
    \mu_*=\mu_*(T_0,K) <\frac{1}{2C(K+M)\exp(CMT_0)},
\]
the assumption (\ref{assump_ext_b}) can be recovered from (\ref{eq:bounds_critcoro_div}). This completes the proof.

\begin{remark}
    Compared with (\ref{eq:critical_assum_1}) in the proof of Theorem~\ref{thm:critical_3d}, here the bootstrap assumption (\ref{eq:assump_ext}) drops the smallness of $\| u^{\epsilon,\delta} - u\|_{L^2}$. Instead, we rely on (\ref{eq:v_L6}) as a compensation.
\end{remark}

\begin{remark} \label{rem:orders}
We close this subsection with a brief discussion on the convergence orders specified in Theorems~\ref{thm:critical_3d} and \ref{thm:critical_3d_inhomo} under the $\mathcal O(1)$ perturbation regime (\ref{eq:O1_Linf}). Denote $v_0=u^{\epsilon,\delta}_0-u_0$. Clearly, to satisfy (\ref{eq:O1_Linf}), we need $\|v_0\|_{H^1}=\mathcal O(\delta^\alpha)$ and $\|v_0\|_{H^2}=\mathcal O(\delta^{-\alpha})$, leaving one free paramter $\alpha\in\mathbb R$ to tune. Since the key tool is the modulated energy $F^{\epsilon,\delta}$ in (\ref{eq:modulated_energy}), it is natural to require all terms in the integrand, in particular $\omega^{\epsilon,\delta}$ and $\sqrt\delta \nabla\omega^{\epsilon,\delta}$, have the same orders in $\delta$; that is, $\alpha = 1/2-\alpha$. Having $\|v_0\|_{H^1}+\sqrt\delta\|v_0\|_{H^2}=\mathcal O(\delta^{1/4})$, we can either set $\|v_0\|_{L^2}=\mathcal O(\delta^{3/4})$ (as in Theorem~\ref{thm:critical_3d}, respecting that the gradient operator $\nabla$ leads to $\mathcal O(\delta^{-1/2})$), or $\|v_0\|_{L^2}=\mathcal O(\delta^{1/4})$ (as in Theorem~\ref{thm:critical_3d_inhomo}, the lowest possible order) as two limiting cases.

It is worth mentioning that, unlike in (\ref{eq:energy_crit_e}), here the energy $E^{\epsilon,\delta}(t)$ in (\ref{eq:energy_crit_inhomo}) does not include the residual function part, because $\|v\|_{L^2}=\mathcal O(\delta^{1/2})$ can already be ensured by the bound of $F^{\epsilon,\delta}$. 
\end{remark}

\subsection{Proof of Theorem~\ref{thm:critical_2d} on $\mathbb T^2$} \label{subsec:thm:critical_2d}

First note that in 2D, the vorticity $\omega=\curl u=\partial_1u_2-\partial_2u_1$ is a scalar. The quantity $A^{\epsilon,\delta}$ in (\ref{eq:omega_ed}) also becomes a scalar:
\begin{equation} \label{eq:A_2d}
    A^{\epsilon,\delta}:=\curl \dive (u^{\epsilon,\delta}\otimes u^{\epsilon,\delta}) = u^{\epsilon,\delta}\cdot \nabla \omega^{\epsilon,\delta}-u^{\epsilon,\delta}\times \nabla \dive u^{\epsilon,\delta} + 2(\dive u^{\epsilon,\delta}) \omega^{\epsilon,\delta},
\end{equation}
which can be verified by a direct calculation. Here it suffices to use the energy $\bar F^{\epsilon,\delta}$ in (\ref{eq:Fbar}), and (\ref{eq:F0_evolve}) still holds for its time evolution.

The proof proceeds with the same bootstrap strategy as in previous sections. Obviously, all involved norms are for $\mathbb T^2$. We will use $C$ as a generic positive constant that depends only on the domain $\mathbb T^2$. 
Denote
\[
    K = 1+\|u\|_{C^0H^2\cap C^1H^1 ([0,T_0]\times \mathbb T^2)}^2.
\]

Besides (\ref{eq:relation_2d}), we now introduce additional bootstrap assumptions
\begin{subequations} \label{eq:assumpt_2d}
\begin{align}
    &\|u^{\epsilon,\delta}(t,\cdot)\|_{L^\infty}^2 \le M=M(C_0,T_0,K), \label{eq:assump_2d_a}\\
    &\|\dive u^{\epsilon,\delta}(t,\cdot)\|_{L^2}^2 + \delta \|\dive u^{\epsilon,\delta}(t,\cdot)\|_{H^1}^2 \le \frac{1}{C^2 M(1+T_0)}. \label{eq:assump_2d_b}
\end{align}
\end{subequations}

Using (\ref{eq:A_2d}) and integration by parts, we obtain
\[
    I:=\int_{\mathbb T^2} (\omega^{\epsilon,\delta} + 2\delta\partial_t\omega^{\epsilon,\delta})A^{\epsilon,\delta}\,dx
    = \int_{\mathbb T^2} \left[ 2\delta\partial_t\omega^{\epsilon,\delta}A^{\epsilon,\delta} + \frac{1}{2}(\dive u^{\epsilon,\delta})(\omega^{\epsilon,\delta})^2 + (\dive u^{\epsilon,\delta}) u^{\epsilon,\delta}\times \nabla \omega^{\epsilon,\delta} \right] \,dx.
\]
Since $\int_{\mathbb T^2} \omega^{\epsilon,\delta}\,dx=0$, we have $\|\omega^{\epsilon,\delta}\|_{L^2(\mathbb T^2)}\le C\|\nabla\omega^{\epsilon,\delta}\|_{L^2(\mathbb T^2)}$ and hence that
\[
\begin{aligned}
    \left| \int_{\mathbb T^2} \frac{1}{2}(\dive u^{\epsilon,\delta})(\omega^{\epsilon,\delta})^2 \,dx \right| &\le C \| \dive u^{\epsilon,\delta} \|_{L^2} \|\omega^{\epsilon,\delta}\|_{L^4}^2 \\
    &\le C \| \dive u^{\epsilon,\delta} \|_{L^2} \|\omega^{\epsilon,\delta}\|_{L^2} \|\omega^{\epsilon,\delta}\|_{H^1}\\
    &\le C \| \dive u^{\epsilon,\delta} \|_{L^2} \|\nabla\omega^{\epsilon,\delta}\|_{L^2}^2,\\
    \left| \int_{\mathbb T^2} (\dive u^{\epsilon,\delta})u^{\epsilon,\delta} \times \nabla \omega^{\epsilon,\delta} \,dx \right| &\le \frac{1}{4}\|\nabla \omega^{\epsilon,\delta}\|_{L^2}^2 + \|u^{\epsilon,\delta}\|_{L^\infty}^2 \|\dive u^{\epsilon,\delta}\|_{L^2}^2,
\end{aligned}
\]
and
\[
\begin{aligned}
    \delta \| A^{\epsilon,\delta} \|_{L^2}^2 &\le 3\delta \| u^{\epsilon,\delta}\|_{L^\infty}^2 \|\nabla\omega^{\epsilon,\delta}\|_{L^2}^2 + 3\delta \|\nabla\dive u^{\epsilon,\delta}\|_{L^2}^2 \|u^{\epsilon,\delta}\|_{L^\infty}^2 + C\delta \|\dive u^{\epsilon,\delta}\|_{L^4}^2 \|\nabla\omega^{\epsilon,\delta}\|_{L^2}^2
\end{aligned}
\]
with
\[
\begin{aligned}
    C\delta\|\dive u^{\epsilon,\delta}\|_{L^4}^2 &\le C^2 \delta^{\frac{1}{2}} \left( \|\dive u^{\epsilon,\delta}\|_{L^2} \cdot\delta^{\frac{1}{2}} \|\dive u^{\epsilon,\delta}\|_{H^1} \right) \\
    &\le \delta^{\frac{1}{2}}\frac{1}{M(1+T_0)}
\end{aligned}
\]
due to (\ref{eq:assump_2d_b}). Then we get
\[
    \left| I \right| \le \int_{\mathbb T^2} |\delta \partial_t\omega^{\epsilon,\delta}|^2\,dx + \left(3\delta M+\delta^{\frac{1}{2}}\frac{1}{M} + \frac{1}{\sqrt M} + \frac{1}{4}\right) \|\nabla\omega^{\epsilon,\delta}\|_{L^2}^2 + \frac{4}{C^2 (1+T_0)}
\]
and hence, by (\ref{eq:F0_evolve}), 
\[
\begin{aligned}
    \frac{d}{dt}\frac{1}{2} \bar F^{\epsilon,\delta}(t) &\le -\int_{\mathbb T^2} \left( \delta \left|\partial_t \omega^{\epsilon,\delta} \right|^2 + \left| \nabla\omega^{\epsilon,\delta} \right|^2 \right) \, dx +|I| \\
    &\le \frac{1}{1+T_0}
\end{aligned}
\]
for sufficiently small $\delta$ and sufficiently large $M$. Moreover, noting again $\partial_t \omega^{\epsilon,\delta}=-\curl \dive U^{\epsilon,\delta}$, by the initial conditions we have
\[
    \bar F^{\epsilon,\delta}(0)\le CC_0.
\]
Therefore, we conclude that
\begin{equation} \label{eq:F0_2d}
    \bar F^{\epsilon,\delta}(t) \le CC_0, \quad t\in[0,T_0],
\end{equation}
which implies (\ref{eq:bounds_crit2d_curl}) once the bootstrap assumption is closed.

Next, we prove the convergence of $u^{\epsilon,\delta}$, using the symmetric hyperbolic structure of (\ref{eq:res}) for $W=(q,v,V)$ defined in (\ref{eq:res_def}). Consider an energy
\[
    E(t) = \|W(t,\cdot)\|_{L^2}^2 \quad \text{with} \quad E(0)\le C\delta
\]
due to the initial conditions. An energy estimate yields (which can be derived similar to (\ref{eq:energy_ineq_btstp}))
\[
    \frac{d}{dt}E(t) \le CE(t) + CK(\epsilon+\delta^2) + \|v\|_{L^4}^4,
\]
where
\[
    \|v\|_{L^4}^4 \le C \|v\|_{L^2}^2 \left( \|v\|_{L^2}^2 + \|\curl v\|_{L^2}^2 + \|\dive v\|_{L^2}^2 \right)
\]
with
\[
    \| \curl v\|_{L^2}^2 \le 2 \left( \|\omega^{\epsilon,\delta}\|_{L^2}^2 + \|\omega\|_{L^2}^2 \right)
    \le 2 \left( 2\bar F^{\epsilon,\delta}(t) + \|\omega\|_{L^2}^2 \right) \le C(C_0+K)
\]
and $\dive v = \dive u$. Using (\ref{eq:assump_2d_b}), we obtain
\[
    \frac{d}{dt}E(t) \le CE(t)^2 + C(C_0+K)E(t) + CK\delta,
\]
which, after a standard argument, leads to
\[
    E(t) \le C e^{CT_0(C_0+K)}(C_0+K)\delta, \quad t\in[0,T_0]
\]
for $\delta< \exp(-CT_0(C_0 + K))$. This immediately gives (\ref{eq:bounds_crit2d_u}).

The last step is to prove the bound for $\dive u^{\epsilon,\delta}$ (\ref{eq:bounds_crit2d_div}) and recover the bootstrap assumption (\ref{eq:assumpt_2d}). The technique is to exploit the symmetric hyperbolic structure of (\ref{eq:relax}). Define
\[
    G^{\epsilon,\delta}(t) = \left \| \partial_t \left( u^{\epsilon,\delta}, \sqrt\epsilon p^{\epsilon,\delta}, \sqrt\delta U^{\epsilon,\delta} \right) \right\|_{L^2}^2 + \delta \left \| \nabla \partial_t \left( u^{\epsilon,\delta}, \sqrt\epsilon p^{\epsilon,\delta}, \sqrt\delta U^{\epsilon,\delta} \right) \right\|_{L^2}^2,
\]
which satisfies
\[
    G^{\epsilon,\delta}(0) \le CK(C_0+K)\delta^{-1},
\]
thanks to the initial conditions. An energy estimate implies that
\[
\begin{aligned}
    \frac{d}{dt}G^{\epsilon,\delta}(t) &\le \| \partial_t (u^{\epsilon,\delta}\otimes u^{\epsilon,\delta}) \|_{L^2}^2 + \delta\| \nabla\partial_t (u^{\epsilon,\delta}\otimes u^{\epsilon,\delta}) \|_{L^2}^2 \\
    &\le C \| u^{\epsilon,\delta} \|_{L^\infty}^2 G^{\epsilon,\delta}(t) + C\delta \|\nabla u^{\epsilon,\delta}\|_{L^4}^2 \|\partial_t u^{\epsilon,\delta}\|_{L^4}^2.
\end{aligned}
\]
The $L^4(\mathbb T^2)$-norms are further estimated as
\[
    \|\partial_t u^{\epsilon,\delta}\|_{L^4}^2 \le C \|\partial_t u^{\epsilon,\delta}\|_{L^2} \|\partial_t u^{\epsilon,\delta}\|_{H^1} \le C\delta^{-\frac{1}{2}} G^{\epsilon,\delta}(t)
\]
and
\[
    \|\nabla u^{\epsilon,\delta}\|_{L^4}^2 \le C \left( \|\dive u^{\epsilon,\delta}\|_{L^2} + \|\omega^{\epsilon,\delta}\|_{L^2} \right) \left( \|\dive u^{\epsilon,\delta}\|_{H^1} + \|\omega^{\epsilon,\delta}\|_{H^1} \right),
\]
where
\[
    \|\omega^{\epsilon,\delta}\|_{L^2}^2 + \delta \|\omega^{\epsilon,\delta}\|_{H^1}^2 \le C \bar F^{\epsilon,\delta}(t) \le CC_0,
\]
thanks to the definition of $\bar F^{\epsilon,\delta}(t)$ (\ref{eq:Fbar}) and (\ref{eq:F0_2d}). Further using bootstrap assumption (\ref{eq:assumpt_2d}), we have
\[
    \frac{d}{dt}G^{\epsilon,\delta}(t) \le C(C_0+M)G^{\epsilon,\delta}(t)
\]
and thus conclude
\[
    G^{\epsilon,\delta}(t) \le CK(C_0+K) e^{C(C_0+M)T_0}\delta^{-1}, \quad t\in[0,T_0].
\]
Noting that
\begin{equation} \label{eq:2d_G_use}
    \|\dive u^{\epsilon,\delta}\|_{L^2}^2 + \delta \|\nabla \dive u^{\epsilon,\delta}\|_{L^2}^2 \le \epsilon G^{\epsilon,\delta}(t),
\end{equation}
we then obtain the desired bound for $\dive u^{\epsilon,\delta}$ in (\ref{eq:bounds_crit2d_div}). To recover (\ref{eq:assump_2d_a}), it holds that (recalling $v=u^{\epsilon,\delta}-u$)
\[
\begin{aligned}
    \|u^{\epsilon,\delta}\|_{L^\infty}^2 &\le \|u\|_{L^\infty}^2 + C\|v\|_{L^2}\|v\|_{H^2} \\
    &\le \|u\|_{L^\infty}^2 + C\|v\|_{L^2} \left( \|v\|_{L^2} + \|\curl v\|_{H^1} + \|\dive u^{\epsilon,\delta}\|_{H^1} \right).
\end{aligned}
\]
Here, we bound $\|v\|_{L^2} \le \sqrt{E(t)}\le \sqrt{M(C_0+K)\delta/C}$ by imposing $M\ge C^2 \exp{(CT_0(C_0+K))}$. With $\|\curl v\|_{L^2}=\mathcal O(1)$, the term $\|\curl v\|_{H^1}$ is dominated by $\|\nabla \curl v\|_{L^2}$ and further by $\|\nabla \omega^{\epsilon,\delta}\|_{L^2} \le \sqrt{\bar F^{\epsilon,\delta}(t)\delta^{-1}} \le \sqrt{CC_0\delta^{-1}}$. On the other hand, (\ref{eq:assump_2d_b}) implies that $\|\dive u^{\epsilon,\delta}\|_{H^1} \le (C\sqrt M)^{-1}\delta^{-1/2}$, with the coefficients $(C\sqrt M)^{-1}$ much smaller than $\sqrt{CC_0}$. Essentially, these estimates lead to
\[
    \|u^{\epsilon,\delta}\|_{L^\infty}^2 \le K+ C\sqrt{\frac{M(C_0+K)}{C}} \sqrt{CC_0} \le 2C \sqrt{C_0M(C_0+K)}
\]
if $M>K^2/(C^2C_0(C_0+K))$. Hence, if we take 
\[
    M=M(C_0,T_0, K) = C^2\max \left\{ 16C_0(C_0+K), \, \exp{(CT_0(C_0+K))}, \, \frac{K^2}{C^4C_0(C_0+K)} \right \}, 
\]
then this ensures $\|u^{\epsilon,\delta}\|_{L^\infty}^2 \le M/2$, a stricter form of (\ref{eq:assump_2d_a}). Finally, with such choice of $M$, we see from (\ref{eq:2d_G_use}) that if a sufficiently small 
\[
    \mu_*(C_0, T_0,K) = \left[ 2C^3MK(C_0+K)(1+T_0)\right]^{-1}e^{-C(C_0+M)T_0} 
\]
is picked, a stricter form of (\ref{eq:assump_2d_b}) is recovered. This closes the bootstrap assumption and hence completes the proof. 

\begin{remark} \label{rem:orders_2d}
In the same spirit as Remark~\ref{rem:orders}, we discuss the convergence orders in Theorem~\ref{thm:critical_2d}. Note that on $\mathbb T^2$, it holds that (see, e.g., Eq. (2.1c) in \cite{HRYZ25})
\[
    \|v_0\|_{L^\infty} \le C \|v_0\|_{L^2} + C \|v_0\|_{L^2}^{\gamma/2} \|\nabla v_0\|_{L^2}^{1-\gamma} \|\nabla^2v_0\|_{L^2}^{\gamma/2}, \quad \forall \ \gamma\in(0,1].
\]
Then, if we assume $\|v_0\|_{L^2}=\mathcal O(\delta^\alpha)$, $\|\nabla v_0\|_{L^2}=\mathcal O(\delta^{1/2+\beta})$ and $\|\nabla^2v_0\|_{L^2}=\mathcal O(\delta^\beta)$ to respect the (modulated) energy structure in (\ref{eq:Fbar}) and (\ref{eq:modulated_energy}), an (overdetermined) problem follows:
\[
\left\{
\begin{aligned}
    &\min\{\alpha, \ \frac{\gamma}{2}\alpha + (1-\gamma)(\frac{1}{2}+\beta)+\frac{\gamma}{2}\beta \}=0, \quad \forall \ \gamma\in(0,1], \\
    & \alpha \ge \frac{1}{2}+\beta.
\end{aligned}
\right.
\]
It is not difficult to show that this problem admits the only solution $\alpha=1/2$ and $\beta=-1/2$.
\end{remark}

\section{Conclusions} \label{sec:concl}

In this paper, we have provided a comprehensive convergence analysis for a two-parameter hyperbolic relaxation system (\ref{eq:relax}) approximating the incompressible Navier-Stokes equations (\ref{eq:ns}) on $\mathbb{T}^d$ ($d=2,3$). For the three-dimensional case $\mathbb T^3$, we proved the simultaneous convergence of both velocity and pressure under $o(1)$ initial velocity perturbations. This was achieved by introducing an intermediate affine system, which successfully decoupled the errors between the approximation and the limit systems, especially for the pressure variable. Furthermore, for $\mathcal O(1)$ initial velocity perturbations, we demonstrated that the velocity field tracks the Navier-Stokes trajectory over its entire lifespan. The convergence is realized within a specific regime where $\epsilon = o(\delta)$ and the initial velocity divergence is small compared to the curl, ensured by a bootstrap argument that preserves this structural hierarchy. Notably, our analysis utilizes a modulated energy functional to overcome the limitations of standard symmetric hyperbolic estimates. We also implement an initial layer correction to handle the stress-like tensor $U^{\epsilon,\delta}$ in the energy estimates to relax the constraints on initial data.

Future research may address the challenge of pressure convergence under $\mathcal O(1)$ perturbations and explore the behavior of the system in more general parameter regimes. Numerical simulations based on (\ref{eq:relax}) also remain a promising direction.

\section*{Acknowledgments}
Funding by the Deutsche Forschungsgemeinschaft (DFG, German Research Foundation) - SPP 2410 \textit{Hyperbolic Balance Laws in Fluid Mechanics: Complexity, Scales, Randomness (CoScaRa)} is gratefully acknowledged.

\bibliographystyle{amsplain}
\bibliography{references}

\end{document}